\documentclass[11pt]{article}
\usepackage[top=1in, bottom=1in, left=1in, right=1in]{geometry}

\usepackage{titletoc}
\titlecontents{section}[1.5em]{\vspace{-1pt}}
    {\thecontentslabel\enspace\enspace}{}{\titlerule*[0.5pc]{.}\contentspage}
    
\usepackage{amsfonts, amsmath, amsthm, xcolor}
\usepackage[breaklinks=true,colorlinks=true]{hyperref}
\usepackage{graphicx}
\newwrite\graphics
\immediate\openout\graphics=\jobname-graphics.txt%
\let\oincludegraphics\includegraphics
\renewcommand{\includegraphics}[2][]{
  \immediate\write\graphics{#2}
  \oincludegraphics[#1]{#2}}
\graphicspath{ {./figures/} }
\newtheorem{theorem}{Theorem}[section]
\newtheorem{proposition}[theorem]{Proposition}

\newtheorem{lemma}[theorem]{Lemma}

\numberwithin{equation}{section}
\newcommand{\N}{\mathbb{N}}
\newcommand{\R}{\mathbb{R}}
\newcommand{\sign}{\operatorname{sign}}
\newcommand{\norm}[1]{\lVert #1 \rVert}

\newcommand{\eps}{\varepsilon}
\newcommand{\nex}{\textup{ne}}
\newcommand{\ex}{\textup{ex}}
\newcommand{\aFF}{a_{\mathrm{\sharp}}}
\newcommand{\aDD}{a_{\mathrm{b}}}
\newcommand{\crm}{}
\newcommand{\nc}{\normalcolor}
\begin{document}
\title{On standing waves of 1D nonlinear Schr\"odinger equation with triple power nonlinearity}
\author{Theo Morrison\thanks{University of British Columbia, Vancouver BC Canada, morrisontgs@gmail.com} \and
Tai-Peng Tsai\thanks{University of British Columbia, Vancouver BC Canada,  ttsai@math.ubc.ca}
}

\date{}
\maketitle

\begin{abstract}
For the one dimensional nonlinear Schr\"odinger equation with triple power nonlinearity and general exponents, we study analytically and numerically the existence and stability of standing waves. Special attention is paid to the curves of non-existence and curves of stability change on the parameter planes.

\end{abstract}

\tableofcontents

\section{Introduction}\label{S1}

Consider the one dimensional nonlinear Schr\"odinger equation with triple power nonlinearity
\begin{equation} \label{eq:nls}
i\partial_t u + \partial_x^2 u + f(u) =0,\quad f(u) = a_1|u|^{p-1}u+a_2|u|^{q-1}u+a_3|u|^{r-1}u
\end{equation}
where $u:\R_t \times \R_x \to \mathbb C$, \crm $a_1,a_3 \in\mathbb R\setminus\{0\}$, $a_2 \in\mathbb R$ and  \nc$1<p<q<r<\infty$.
Our primary goal is to study the existence and stability properties of standing waves of \eqref{eq:nls} with the coefficients being the parameters.
This paper is a continuation of our previous study \cite{LiTsZw21} in which we focused on the special case $(p,q,r)=(2,3,4)$.

Nonlinear Schr\"odinger equations appear in many areas of physics such as nonlinear optics (see e.g.~\cite{Ag07}) or Bose-Einstein condensation. 
Mathematically, they form one of the primary examples of dispersive partial differential equations.
The Cauchy problem for \eqref{eq:nls} with general $f(u)$ is well known (see \cite{Ca03} and the references therein) to be well-posed in the energy space $H^1(\mathbb R)$: for any $u_0\in H^1(\mathbb R)$, there exists a unique maximal solution $u\in C((-T_*,T^*),H^1(\R)) \cap C^1((-T_*,T^*),H^{-1}(\R))$ of \eqref{eq:nls} such that $u(t=0)=u_0$. Moreover, the energy $E$ and the mass $Q$, defined by
\[
E(u)=\frac12\norm{u_x}_{L^2}^2-\int_\R F(|u|)\,dx, 
\quad Q(u)=\norm{u}_{L^2}^2,
\] 
where $F(t)=\int_0^{t} f(s)\,ds$,
are conserved along the flow and the blow-up alternative holds (i.e.~if $T^*<\infty$ (resp. $T_*<\infty$), then $\lim_{t\to T^* \text{ (resp }-T_*{\text{)}}}\norm{u(t)}_{H^1}=\infty$).

A \emph{standing
wave} is a solution of \eqref{eq:nls} of the form $u(t,x) = e^{i \omega t}\phi(x) $ for some
$\omega \in \R$ and a {nonzero} \emph{profile} $\phi \in C^2(\R){\cap H^1(\R)}$, which then satisfies
\begin{equation} \label{eq:ode}
\phi'' =\omega \phi -f(\phi).
\end{equation}
We only consider real-valued $\phi$ in this paper.
Standing waves and more general 
solitary waves are the building blocks for the nonlinear dynamics of \eqref{eq:nls}, as it is expected that, generically, a solution of \eqref{eq:nls} will decompose into a dispersive linear part and a combination of nonlinear structures as solitary waves. This vague statement is usually referred to as the \emph{Soliton Resolution Conjecture}.
Therefore, understanding the dynamical properties of standing waves, in particular their stability, is a key step in the analysis of the dynamics of \eqref{eq:nls}.  Several stability concepts are available for standing waves. The most commonly used is \emph{orbital stability}, which is defined as follows. A standing wave $e^{i\omega t}\phi(x)$ solution of \eqref{eq:nls} is said to be \emph{orbitally stable} if for any $\eps>0$, there exists $\delta>0$ such that if $u_0\in H^1(\R)$ verifies
\[
\norm{u_0-\phi}_{H^1}<\delta,
\]
then the associated solution $u$ of \eqref{eq:nls} exists globally and verifies 
\[
\sup_{t\in\R}\inf_{y\in\R,\theta\in\R}\norm{u(t)-e^{i\theta}\phi(\cdot-y)}_{H^1}<\eps.
\]
In the rest of this paper, when we talk about stability/instability, we always mean \emph{orbital} stability/instability.

  The groundwork for orbital stability studies was laid down by Berestycki and Cazenave \cite{BeCa81}, Cazenave and Lions \cite{CaLi82} and Weinstein \cite{We83,We85}. Two approaches lead to stability or instability results: the variational approach of \cite{BeCa81,CaLi82}, which exploits global variational characterizations combined with conservation laws or the virial identity, and the spectral approach of \cite{We83,We85}, which exploits spectral and coercivity properties of linearized operators to construct a suitable Lyapunov functional.  Later on, Grillakis, Shatah and Strauss \cite{GrShSt87,GrShSt90} developed an abstract theory which, under certain assumptions, boils down the stability study of a branch of standing waves $\omega\to \phi_{\omega}$  to the study of the sign of the quantity
\(
\frac{\partial}{\partial \omega} Q(\phi_{\omega}).
\)
Note that the theory of Grillakis, Shatah and Strauss has known recently a considerable revamping in the works of De Bi\`evre, Genoud and Rota-Nodari \cite{DeGeRo15,DeRo19}.

With the above mentioned techniques, the orbital stability of positive standing waves has been completely determined in the single power case $f(u) = a_1|u|^{p-1}u$ in any dimension $d\geq1$ in \cite{BeCa81,CaLi82, We83,We85}. In this case, positive standing waves exist if and only if $a_1>0$, $\omega>0$, 
{and $1<p<p_{\text{max}}$, where $p_{\text{max}}=\infty$ for $d=1,2$ and $p_{\text{max}}=1+\frac{4}{d-2}$ for $d\ge 3$. When they exist,}
they are stable if $1<p<1+\frac4d$ (i.e.~$1<p<5$ in dimension $d=1$), and unstable if $1+\frac4d\leq p<p_{\text{max}}$ (i.e.~$5\leq p<\infty$ in dimension $d=1$). The scaling property of the single power nonlinearity plays an important role in the proof and ensures in particular that stability and instability are independent of the value of the frequency $\omega$. 
It turns out that there is no scaling invariance for multiple power nonlinearities, which makes the stability study more delicate. As a matter of fact, only very {limited} partial results are available so far in higher dimensions. In dimension $1$, the situation is a bit more favorable, as one might exploit the ODE structure of the profile equation \eqref{eq:ode} in the analysis.

Preliminary investigations for the stability of standing waves in dimension $1$ were conducted by Iliev and Kirchev \cite{IlKi93} in the case of a generic nonlinearity. In particular, a formula for the slope condition was obtained in \cite{IlKi93}; {see Theorem \ref{thm:stab}.} 
The stability of standing waves for double power nonlinearity in dimension $1$ was initiated by Ohta \cite{Oh95} and continued by Maeda \cite{Ma08} and 
Fukaya and Hayashi \cite{FuHa21}.
The remaining cases were completely classified in Kfoury, Le Coz and Tsai \cite{MR4480890}.
Hayashi \cite[Theorem 1.3]{Ha21}  is similar to \cite{MR4480890} but it does not include the cases $1<p<9/5$. See \cite[Theorem 1]{MR4480890} for a detailed description.

For the triple power case as in \eqref{eq:nls}, very little is known. In our previous study \cite{LiTsZw21}, we focused on the special case $(p,q,r)=(2,3,4)$. Many results of \cite{LiTsZw21} will be shown to persist for general {triple power} $f(u)$, but we will also see new phenomena.
When $a_1<0$ and $a_3>0$, we say that the nonlinearity is \emph{defocusing-focusing}, or DF, with analogous definitions for other possible signs combinations, with a total of 4 cases FF, FD, DF and DD. Note that there is no DD case for double power nonlinearity {(corresponding to $a_2=0$)} as there is no standing wave when all coefficients are negative.
For a solution $u$ of the  NLS (\ref{eq:nls}), we may consider $u(x,t) = \kappa v(\lambda^{-1}x,\lambda^{-2}t)$ for some $\kappa,\lambda>0$. Then $v$ satisfies
\[
i\partial_t v +\partial_x^2 v + b|v|^{p-1}v+c|v|^{q-1}v+d|v|^{r-1}v=0,
\]
with
\[
b=a_1\kappa^{p-1}\lambda^2,\quad c = a_2\kappa^{q-1}\lambda^2\quad d=a_3\kappa^{r-1}\lambda^2.
\]
Choosing $\kappa = |a_1/a_3|^{1/(r-p)}$ and $\lambda=(|a_3|/|a_1|^\frac{r-1}{p-1})^\frac{p-1}{2(r-p)}$ gives $|b|=|d|=1$. Since $u$ and $v$ have the same qualitative properties, we may assume that $|a_1|=|a_3|=1$ without loss of generality. For the rest of this paper, we consider $a_1=\pm1$, $a_2=-\gamma$, $a_3=\pm1$ for $\gamma\in\R$.

To describe our results, we need a few definitions. The parameter domain for 
$(\omega,\gamma)$ is the half-plane $\Omega=(0,\infty)\times \R$.
In each of the 4 cases FF, FD, DF, and DD, we denote the subset of $(\omega,\gamma)\in\Omega$ 
for which a standing wave solution exists by $R_{\ex}$. We denote the boundary of $R_{\ex}$ in $\Omega$ by $\Gamma_{\nex}$ (not including the $\gamma$-axis).
When the standing wave $\phi_{\omega}=\phi_{\omega,\gamma}$ exists, we define
the \emph{stability functional}
\begin{equation}
J(\omega,\gamma) = \frac{\partial}{\partial \omega} Q(\phi_{\omega,\gamma}) =\frac{\partial}{\partial \omega}  \int_\R \phi_{\omega,\gamma}^2(x)\,dx, \quad (\omega,\gamma)\in R_{\ex}.
\end{equation}
As is well known in the stability theory \cite{GrShSt87,GrShSt90} and mentioned previously, under certain assumptions, the sign of $\frac{\partial}{\partial \omega} Q(\phi_\omega)$
determines stability. For our 1D NLS, it follows from Iliev-Kirchev \cite{IlKi93}
that $e^{i\omega t}\phi_ \omega (x)$ is stable when 
$J(\omega,\gamma)>0$, and unstable  when 
$J(\omega,\gamma)<0$; see Lemma \ref {thm:stab}. Because of this, 
the zero level curve of $J$ is of particular interest since it
 is where $J$ changes sign, and indicates the change of the stability property.
The curve of nonexistence $\Gamma_{\nex}$  exists in the FF, FD and DD cases but not in the DF case. As to be shown in Proposition \ref{prop:Rex},
when $\Gamma_{\nex}$ exists, it can be parameterized by a decreasing function
\begin{equation}\label{1.4}
{\omega=\omega^*(\gamma)}
\end{equation}
where $\gamma_1\le \gamma<\infty$, $\gamma \in \R$ and 
$-\infty<\gamma<\gamma_1$,  in the FF, FD and DD cases, respectively.
The two values of $\gamma_1$ for FF and DD cases are different.

{We highlight a few interesting new phenomena observed numerically:
\begin{enumerate}
\item In the FF case with powers 6,7,8, 
for some fixed $\gamma$, the standing wave family $\phi_\omega$ is defined for all $\omega>0$ and is of type USU, i.e., $\phi_\omega$ changes from being unstable to stable and then back to unstable when $\omega$ increases from $0$ to $\infty$. 
Recall we have type SUS for powers 2,3,4 in \cite{LiTsZw21}.
See Figure \ref{FFcase} in Section \ref{S2}. Also see Section \ref{S6} for related propositions.

\item In the FD case with powers $p=3$ and $r=7$, there is no unstable region when $q=5$, while there is an unstable region when $q=6$ for $\gamma$ sufficiently negative. In the latter case, stability change occurs twice for fixed, sufficiently negative $\gamma$, and it is of type SUS. 
See Figure \ref{FDcase}. Also see Section \ref{S7} for related propositions.

\item In the DF case, the standing waves may be all stable, all unstable, or have stability change. For powers 2,2.5,3, stability change occurs at most once for fixed $\gamma$. For powers 3,4,7, stability change occurs twice for fixed, sufficiently negative $\gamma$, and it is of type USU. See Figure \ref{DFcase}. Also see Section \ref{S8} for related propositions.

\item In the DD case, the standing waves may be all stable, or have stability change. When there is stability change, both $\Gamma_{\nex}$ and the stability change curve start from the $\gamma$-axis, and the starting value of $\gamma$ may or may not be the same. See Figure \ref{DDcase}. Also see Section \ref{S9} for related propositions.
\end{enumerate}
}

Some of these numerical observations are proved rigorously.
Among the theoretical results, we parameterize the non-existence curve $\Gamma_\nex$ in all cases (FF, FD and DD) in Proposition \ref{prop:Rex}, study the limit of $J(\omega,\gamma)$ as $(\omega,\gamma) \to \Gamma_\nex$  (Proposition \ref{prop:Gamlimits}),
 as $\omega \to 0^+,\infty$ or as $\gamma \to \pm \infty$ (Propositions \ref{prop:FFlimits}, \ref{prop:FDlimits}, \ref{prop:DFlimits}, \ref{prop:DDlimits} and \ref{prop:DDJ0limits}), and identify regions of $p,q,r,\omega,\gamma$ where $J$ has a fixed sign (Propositions \ref{prop:FFbddstab}, \ref{prop:FDallstab}, \ref{prop:DFJ0pos}--\ref{prop:DFallunStab}). 
 
 Properties of the Beta function (including upper and lower bounds of $\partial_x B(x,\frac12)$ in Lemma \ref{lem:Bxineq}) is used  to calculate the limiting sign of $J(\omega,\gamma)$ as $\omega \to 0^+$ in the D* cases. Descartes' rule of signs for real exponents (Lemma \ref{lem:RoS}) is used repeatedly for existence and other occasions that we need to count sign changing.

In the rest of this paper, we first describe our numerical observations in Section \ref{S2}. 
We then give preliminary results in Section \ref{S3}. 
We consider the existence of standing waves in  Section \ref{S4}, and the limit of $J(\omega,\gamma)$ near $\Gamma_{\nex}$ in Section \ref{S5}.
We state theorems and give detailed proofs for each of the 4 cases FF, FD, DF, and DD in Sections \ref{S6}--\ref{S9}.

\section{Numerical observations}\label{S2}

In this section we present diagrams of the parameter half plane in $\omega$, $\gamma$ for some values of $p,q,r$. The diagrams were generated in MATLAB by evaluating $J(\omega,\gamma)$ on a mesh, and then using the MATLAB contour function to approximate level curves of $J$. The formula (\ref{I-KJ}) was used to evaluate $J$, and the integral in this formula was evaluated using the MATLAB function quadgk. In each diagram, $\Gamma_{\nex}$ is drawn in black and the zero level curve of $J$ is drawn in blue.

In the diagrams for the FF case, the nonexistence curve $\Gamma_{\nex}$ exists and is of the form \eqref{1.4} for {$\gamma_1<\gamma<\infty$.} 
We have $\lim _{\omega \to \omega^*(\gamma)^-}J(\omega,\gamma)= \infty $ and
$\lim _{\omega \to \omega^*(\gamma)^+}J(\omega,\gamma)= -\infty$ on the left and right sides of $\Gamma_{\nex}$
(Propositions \ref{prop:Rex} and \ref{prop:Gamlimits}).

By mean value theorem, in a neighborhood of  the endpoint $(\omega^*(\gamma_1),\gamma_1)$ of $\Gamma_{\nex}$, it emanates a branch of the level curve 
$J=c$ 
for every $c \in \R$, which 
appears to have the same slope as $\Gamma_{\nex}$ at the endpoint.
For a few choices of powers (diagrams 2-6 of Figure \ref{FFcase}), the continuation of these branches occupy the entire $\omega$-$\gamma$ half plane. However, in diagrams 1, 7 and 8 of Figure \ref{FFcase}, some level curves $J=c$ have two branches.

For powers $1.5,2,2.25$ and $1.5,2,2.75$, the zero level curve $J=0$ turns upwards and back towards the nonexistence curve $\Gamma_{\nex}$
from the right. The level curves $J=c$ with $c<0$ are squeezed between the zero level curve and $\Gamma_{\nex}$.
The behavior of the level curves $J=c$ with $c>0$ are different for these two diagrams: 
In the first diagram with powers $1.5,2,2.25$, each of them appears to have two branches. 
For $0<c<60$, one branch emanates from the endpoint $(\omega^*(\gamma_1),\gamma_1)$ and  turns upwards and back towards $(0,\infty)$ along the right side of $J=0$, and another branch starts from $(0,\infty)$ and
slides down along the left side of $\Gamma_{\nex}$ before eventually going to
$\gamma\to -\infty$. For $c>100$, one branch emanates from the endpoint and  turns clockwise towards $(0,\infty)$ along the left side of $\Gamma_{\nex}$, and another branch slides down 
from $(0,\infty)$ along the right side of $J=0$ before eventually going to
$\gamma\to -\infty$. These are consistent with Proposition \ref{prop:FFlimits} (1f, 2e, 3a, 4a) on the limit of $J$ as $\gamma \to \pm \infty$ or as $\omega \to 0,\infty$.

In the second diagram with powers $1.5,2,2.75$, $J=c$ with $c>0$ has only one branch which emanates from the endpoint and  turns clockwise towards $(0,\infty)$ along the left side of $\Gamma_{\nex}$, consistent with Proposition \ref{prop:FFlimits} (1f, 2c, 3c, 4a).

For powers $2,3,4$ the zero level curve turns upwards, but
maintains a positive slope and 
 does not have a maximum $\omega$ value. 
 Level curves $J=c$ emanate from the endpoint and occupy the entire $\omega$-$\gamma$ half plane. Those with $c>0$ turns clockwise towards $(0,\infty)$ along the left side of $\Gamma_{\nex}$, the same as powers $1.5,2,2.75$.
 This is 
consistent with Proposition \ref{prop:FFlimits} (1f, 2c, 3f, 4a). 
 
For powers $3,4,5$, the curve $J=0$ appears to approach the $\omega$ axis as $\omega \to\infty$, and for powers $3,4,7$, the curve goes downwards  as $\omega \to\infty$. 
In both cases, level curves $J=c$ emanating from the endpoint occupy the entire $\omega$-$\gamma$ half plane, those with
$c<0$ go upward as $\omega \to\infty$, and those with $c>0$ go clockwise, some with large $c$ follow $\Gamma_{\nex}$ for a while, but all eventually go downward toward $(0,-\infty)$. This is 
consistent with Proposition \ref{prop:FFlimits} (1d, 2ba, 3f, 4a). 

For powers $3,6,7$ the curve $J=0$ turns down and back towards the $\gamma$ axis, which illustrates the uniform bound on the stable region given in Proposition \ref{prop:FFbddstab}. The level curves $J=c\not=0$ have similar behavior as those for powers $3,4,7$, consistent with Proposition \ref{prop:FFlimits} (1d, 2a, 3f, 4b). 

For powers $5,6,7$, the curve $J=0$ turns down and back towards the $\gamma$ axis, and 
approach $(0,0)$ as $\omega \to 0$. The level curves $J=c$ for $c>0$ emanate from the endpoint and
turns toward $(0,0)$. The level curve $J=-100$ has two branches, one emanates from the endpoint and goes up, the other emanates from $(0,0)$ and goes downward. The level curve $J=-1$ also has two branches, one emanates from the endpoint and turns toward $(0,0)$ below $J=0$, the other is on the right of the end point $(\omega^*(\gamma_1),\gamma_1)$ and has a point for every $\gamma\in\R$.  This is consistent with Proposition \ref{prop:FFlimits} (1b(iv), 2a, 3f, 4b). 
 
Finally, for powers $6,7,8$, the curve $J=0$ turns down and back towards the $\gamma$ axis, and 
turns upwards toward $(0,\infty)$. The level curves $J=c$ for $c>0$ emanate from the endpoint and are squeezed between the curve $J=0$ and $\Gamma_{\nex}$. 
 The level curve $J=-10000$ has two branches, one emanates from the endpoint and goes up on the right side of $\Gamma_{\nex}$ , the other slides down from $(0,\infty)$ from the left side of $J=0$ and goes downward. The level curve $J=-1$ also has two branches, one emanates from the endpoint and turns clockwise, going toward $(0,\infty)$ along the left side of $J=0$, the other is on the right of the end point $(\omega^*(\gamma_1),\gamma_1)$ and has a point for every $\gamma\in\R$.  This is consistent with Proposition \ref{prop:FFlimits} (1a, 2a, 3f, 4b).

\begin{figure}[htp]
\includegraphics[scale = .5]{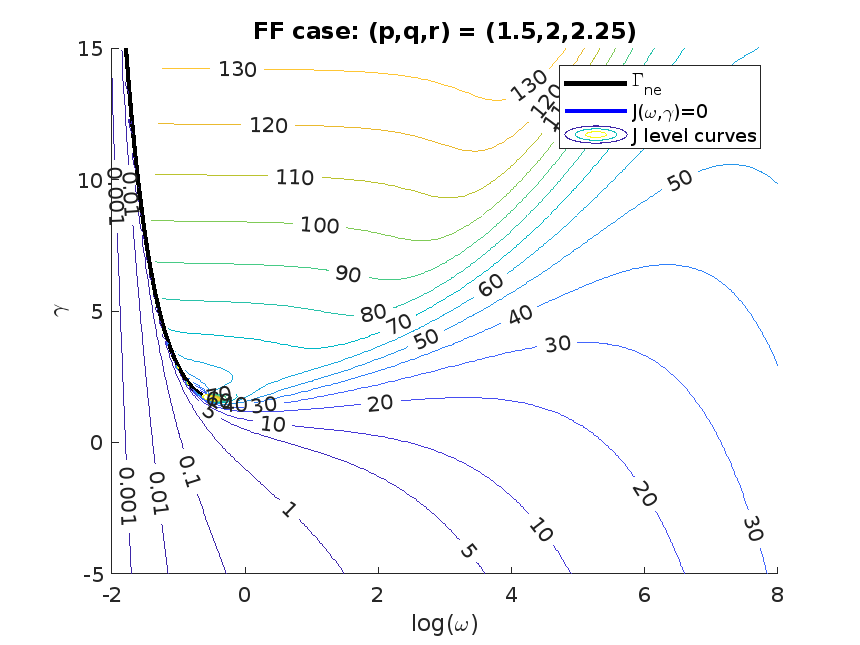}
\includegraphics[scale = .5]{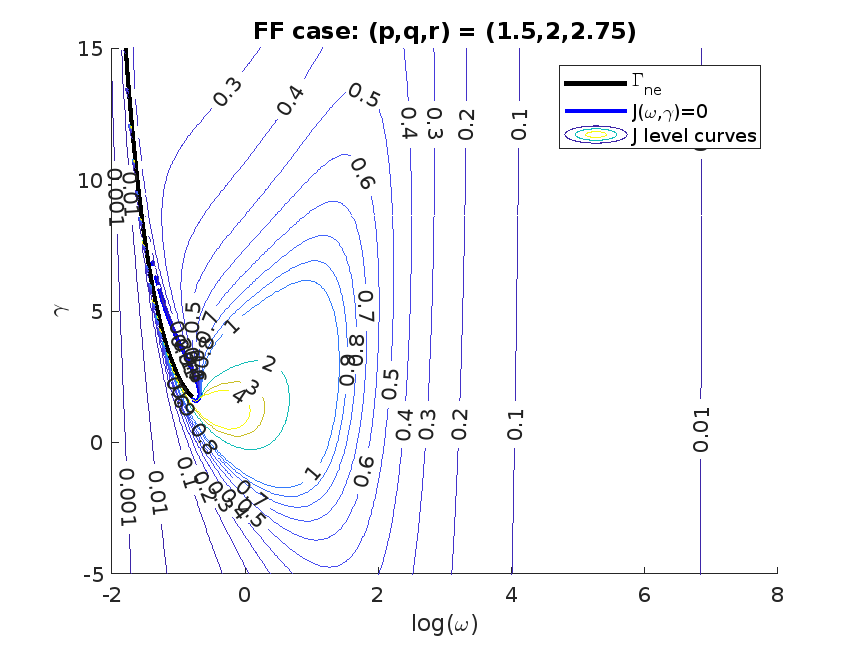}
\includegraphics[scale = .5]{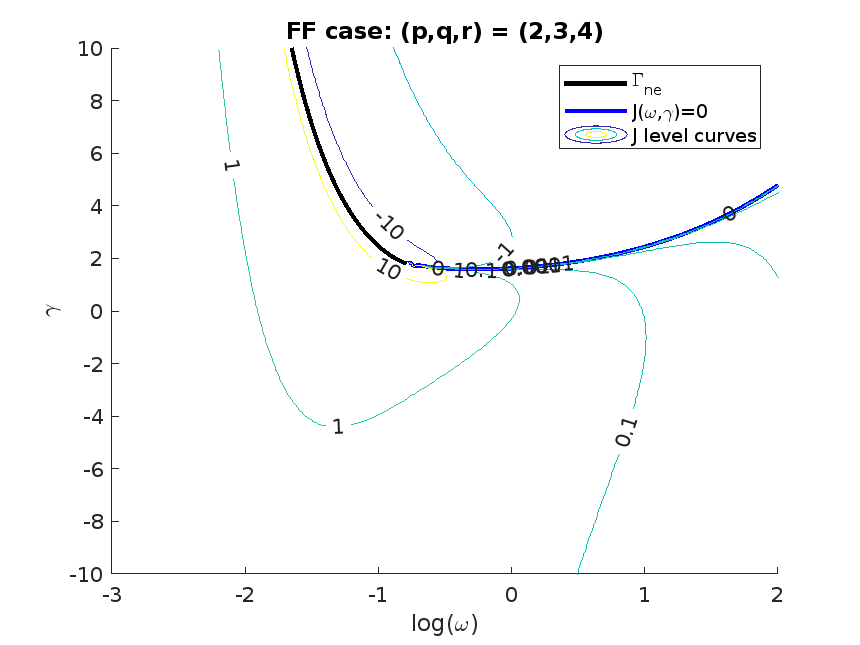}
\includegraphics[scale = .5]{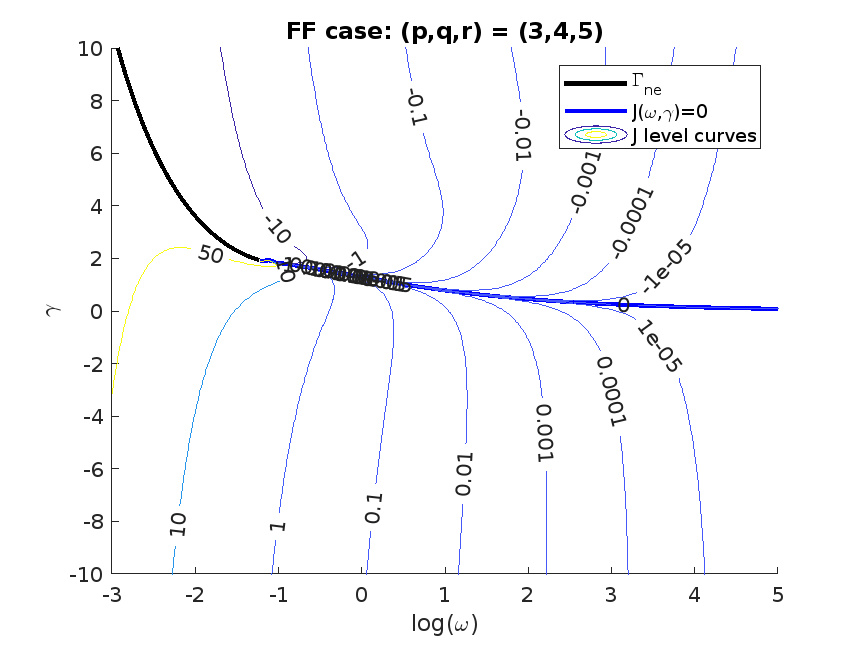}
\includegraphics[scale = .5]{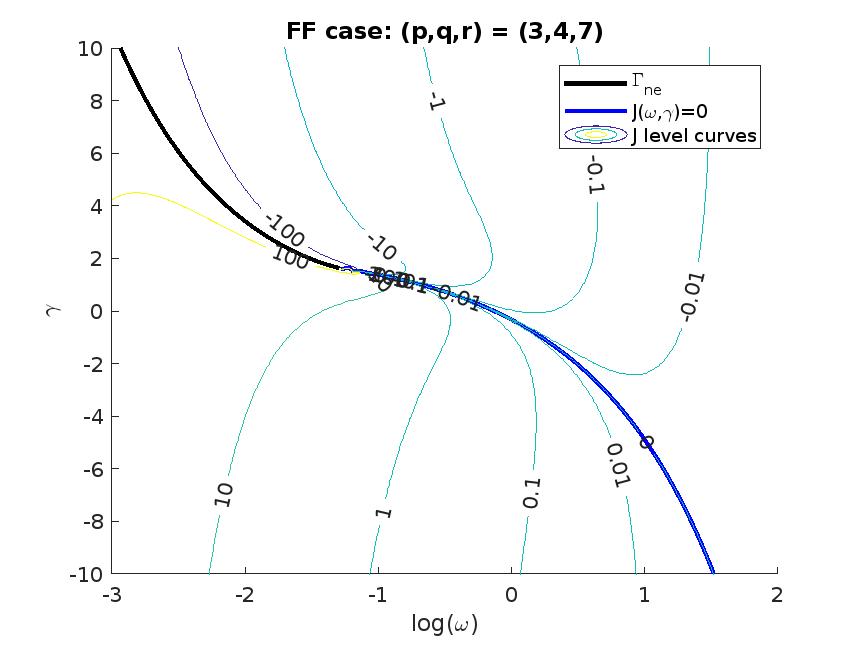}
\includegraphics[scale = .5]{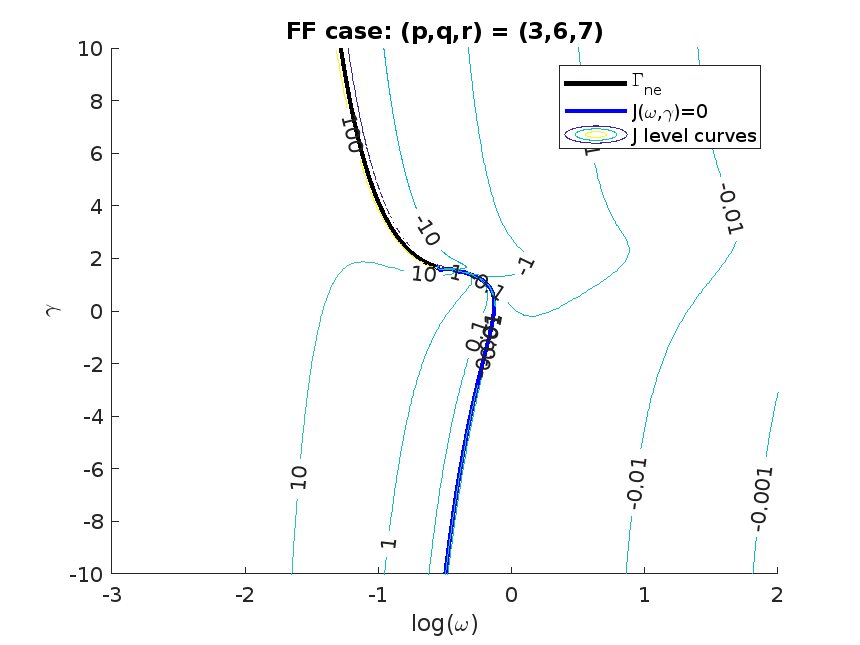}
\includegraphics[scale = .5]{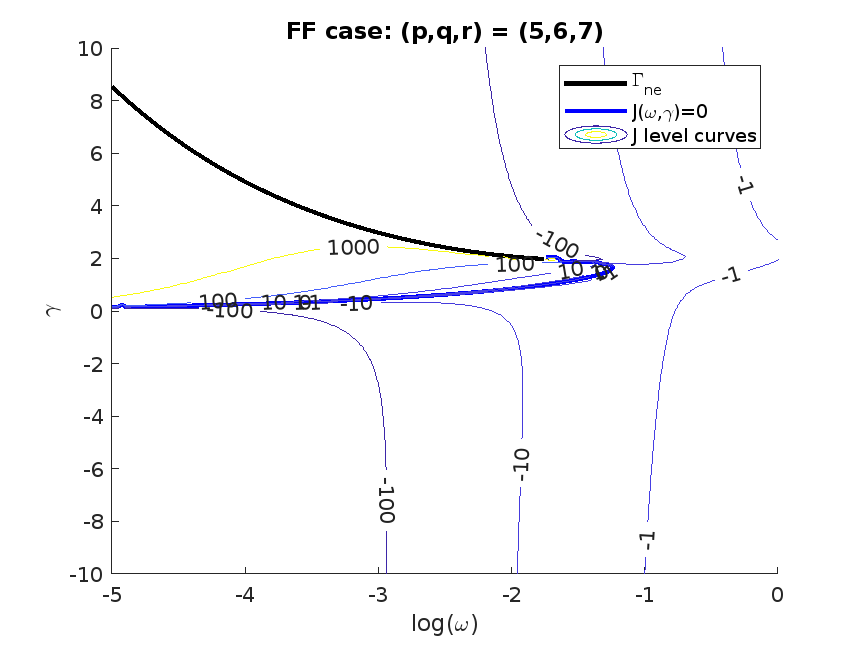}\hfill
\includegraphics[scale = .5]{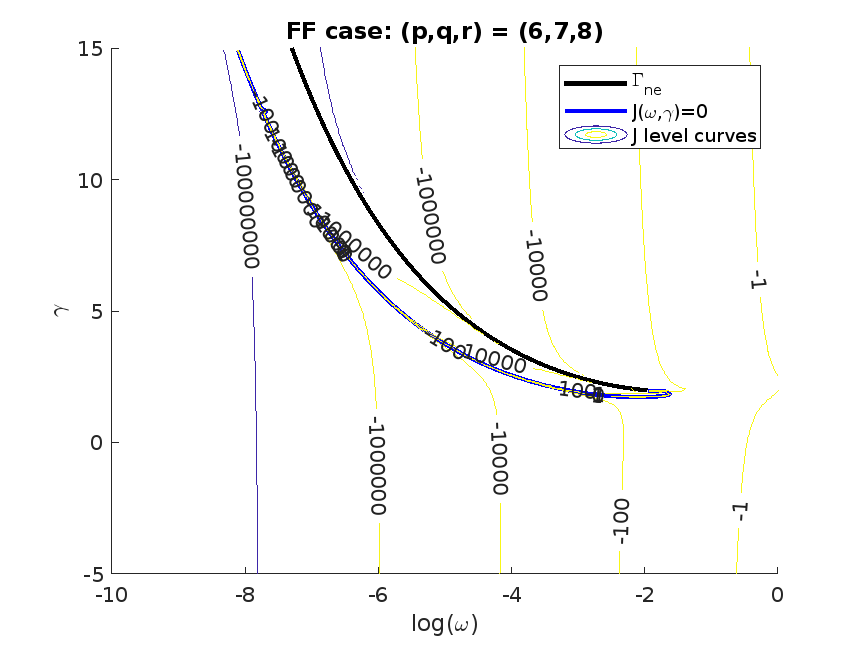}
\caption{FF case}\label{FFcase}
\end{figure}

\newpage
In the diagrams for the FD case, {$\omega^*(\gamma)$ in \eqref{1.4} exists for $\gamma\in\R$,}
$\omega^*(\gamma)\to 0$ as $\gamma\to\infty$, and $\omega^*(\gamma)\to\infty$ as $\gamma\to-\infty$ (cf.~Proposition \ref{prop:Rex} part 2)\nc. We see the existence of an unstable region depending on the value of $q$: For powers $3,5,7$, $J(\omega,\gamma)>0$ for all $(\omega,\gamma)\in R_{\ex}$,  (cf.~Proposition \ref{prop:FDallstab}). For the other 3 diagrams,  there is an unstable region with $J<0$ for sufficiently large $-\gamma$, 
consistent with Proposition \ref{prop:FDlimits} part 2(b).

For powers $3,6,7$, the unstable region stays away from the $\gamma$-axis, consistent with $\lim_{\omega \to 0^+} J(\omega,\gamma)=\infty$ (Proposition \ref{prop:FDlimits} part 1(d)).

For powers $5,6,7$, the zero level curve $J=0$ meets the $\gamma$-axis at $(0,0)$, and the
 unstable region is bordered by the zero level curve and the \emph{negative} $\gamma$-axis. This is consistent with $\lim_{\omega \to 0^+} J(\omega,\gamma)=\text{sign}(\gamma)\infty$ (Proposition \ref{prop:FDlimits} part 1(b)iv).

For powers $5.5, 6,7$, the zero level curve $J=0$ seems to contain a point $(\omega(\gamma),\gamma)$ for all $\gamma \in \R$, and the unstable region is bordered by it and the $\gamma$-axis. This is consistent with $\lim_{\omega \to 0^+} J(\omega,\gamma)=-\infty$ (Proposition \ref{prop:FDlimits} part 1(a)).

\begin{figure}[htp]
\includegraphics[scale = .5]{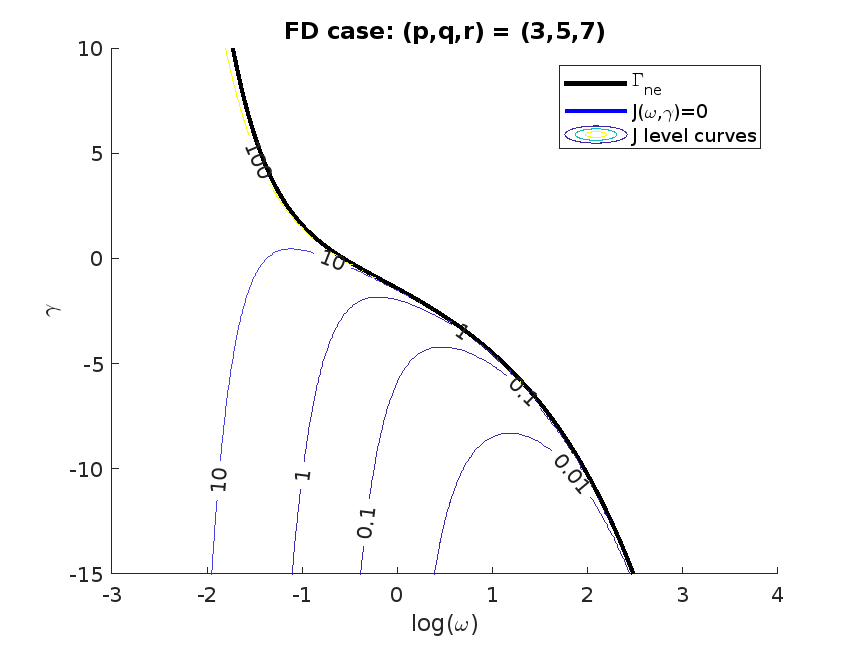}
\includegraphics[scale = .5]{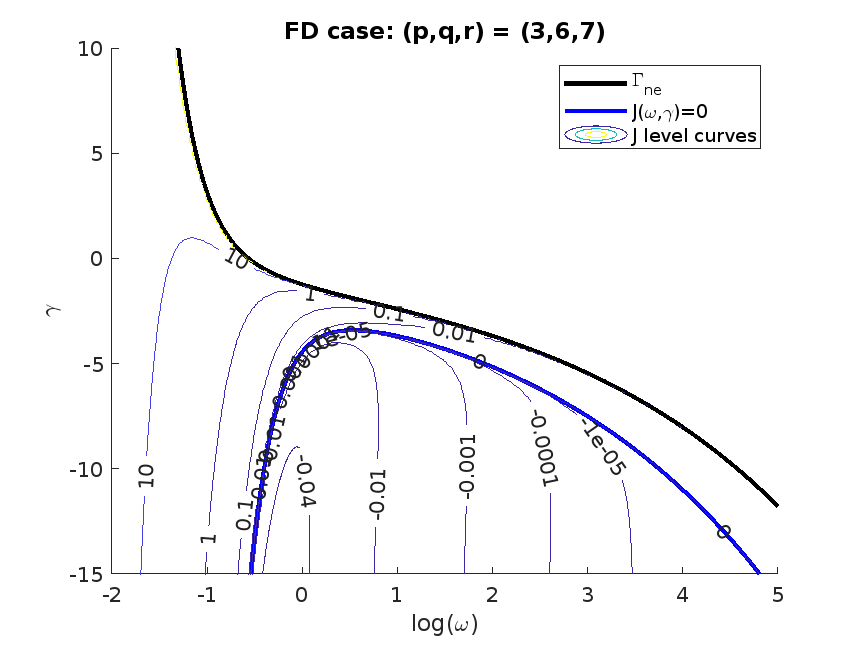}
\includegraphics[scale = .5]{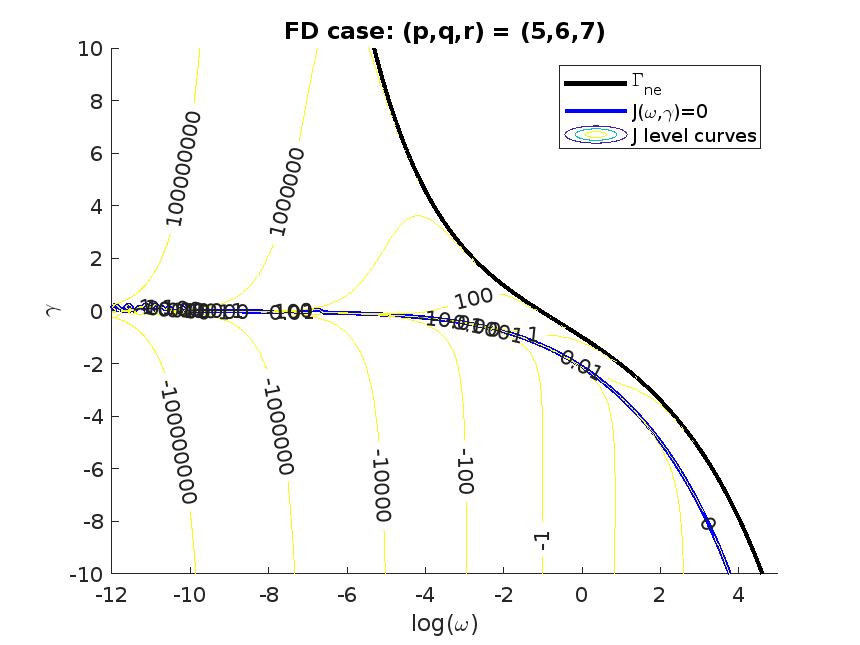}\hfill
\includegraphics[scale = .5]{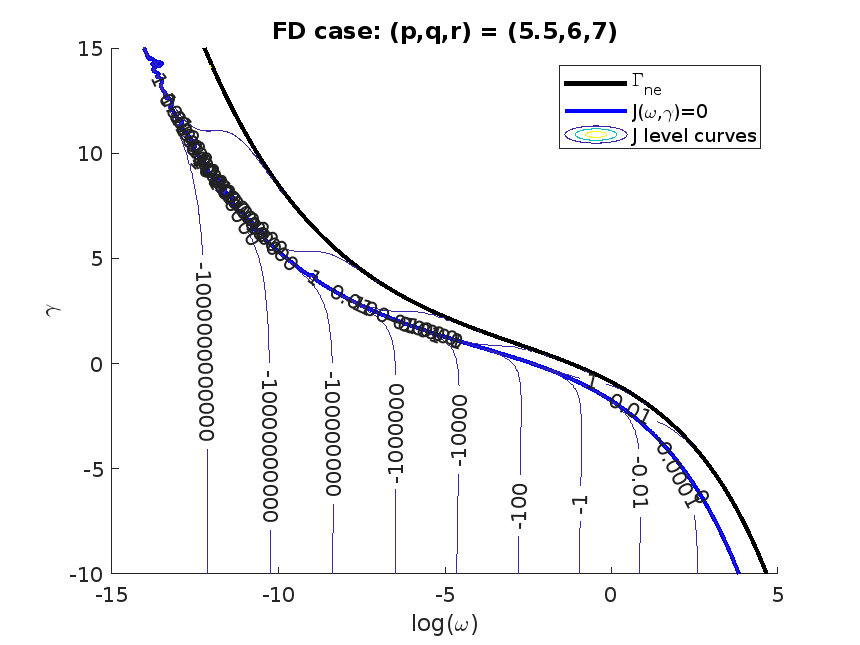}
\caption{FD case}\label{FDcase}
\end{figure}

\newpage
In the DF case, there is no nonexistence curve {and $\phi$ exists for all $(\omega,\gamma) \in \Omega$}. For powers $1.5,2,3$ solutions are stable for all $\omega,\gamma$. Indeed, for all $p,q,r$ that we tested numerically, all solutions appear to be stable when $2q+r\leq7$. This is expected, since the limits of $J$ in Propositions \ref{prop:DFlimits} and \ref{prop:DFJ0pos} are all positive when $2q+r\leq 7$. For powers $2,2.5,3$ we have $2p+q<7<2q+r$, so Proposition \ref{prop:DFJ0US} shows that $J(0,\gamma)>0$ for sufficiently large $-\gamma$, and $J(0,\gamma)<0$ for sufficiently large $\gamma$. We see that this is the case in the diagram for $2,2.5,3$, and the zero level curve appears to have a finite limit as $\omega\to 0$. For powers $3,4,7$, the zero level curve does not meet the $\gamma$ axis, but solutions are stable for large $-\gamma$ as $q<5$, \crm (cf.~Propositions \ref{prop:DFlimits} and \ref{prop:DFJ0neg})\nc. For $3,5,7$ it appears that the {$J$ function is always} negative, which is consistent with Proposition \ref{prop:DFallunStab}.
\begin{figure}[htp]
\includegraphics[scale = .5]{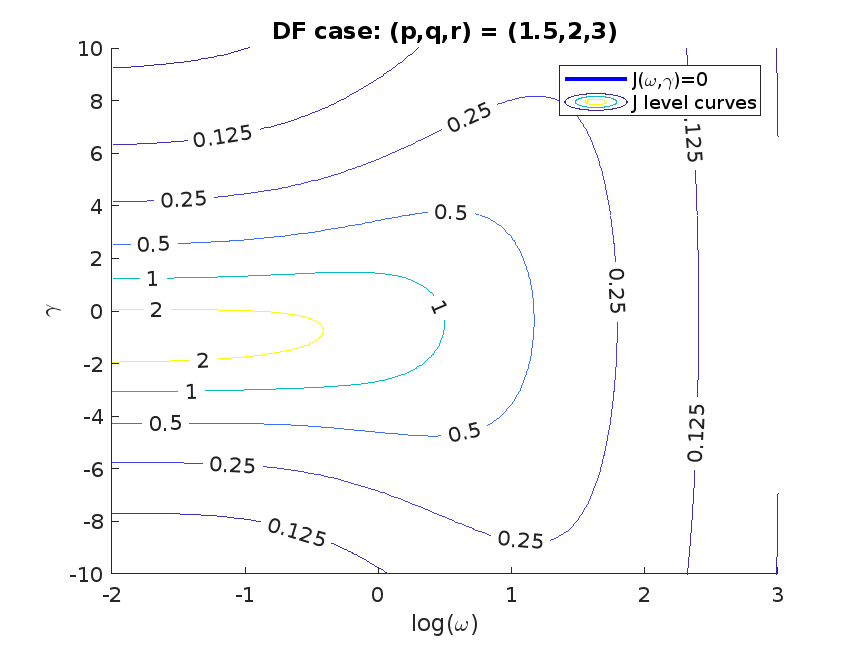}
\includegraphics[scale = .5]{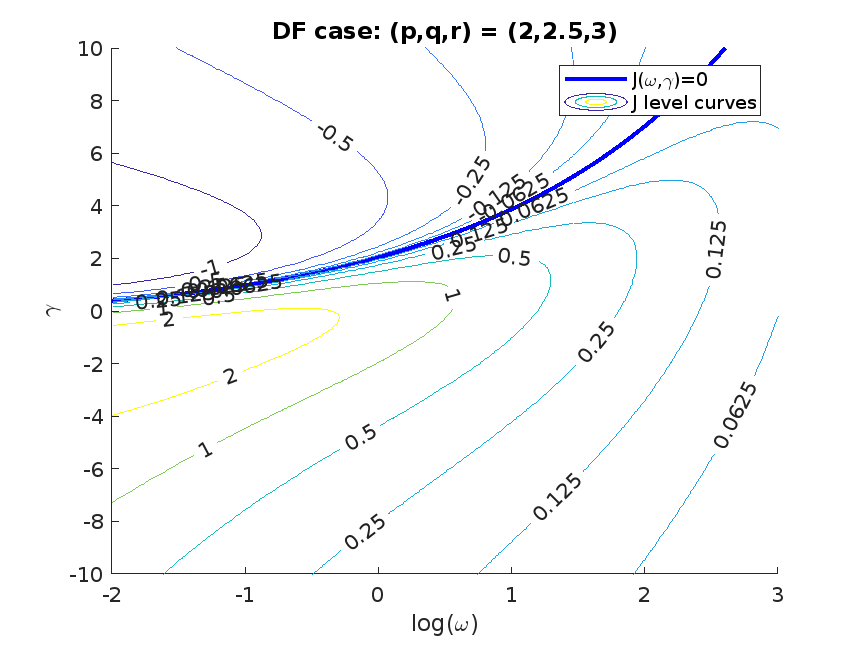}
\includegraphics[scale = .5]{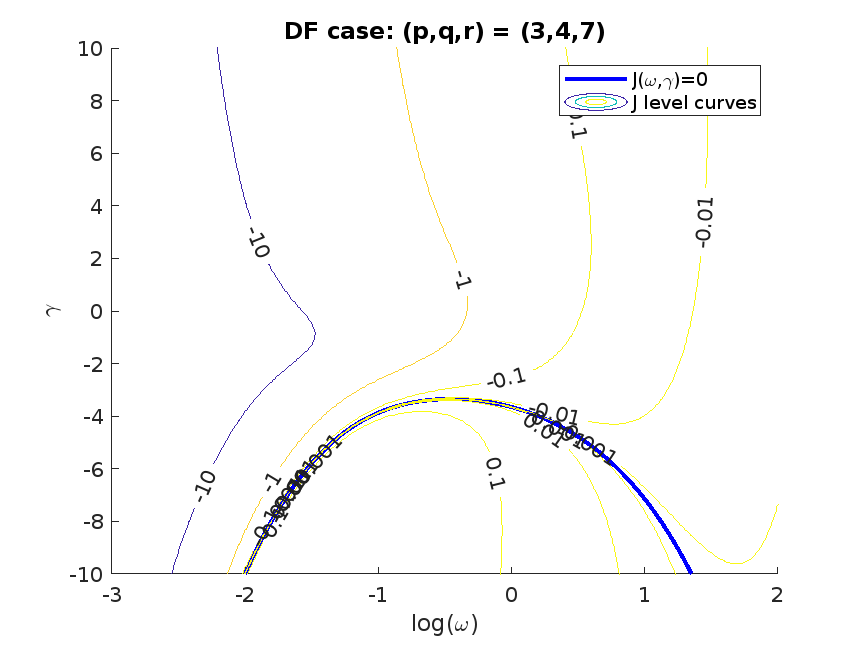}\hfill
\includegraphics[scale = .5]{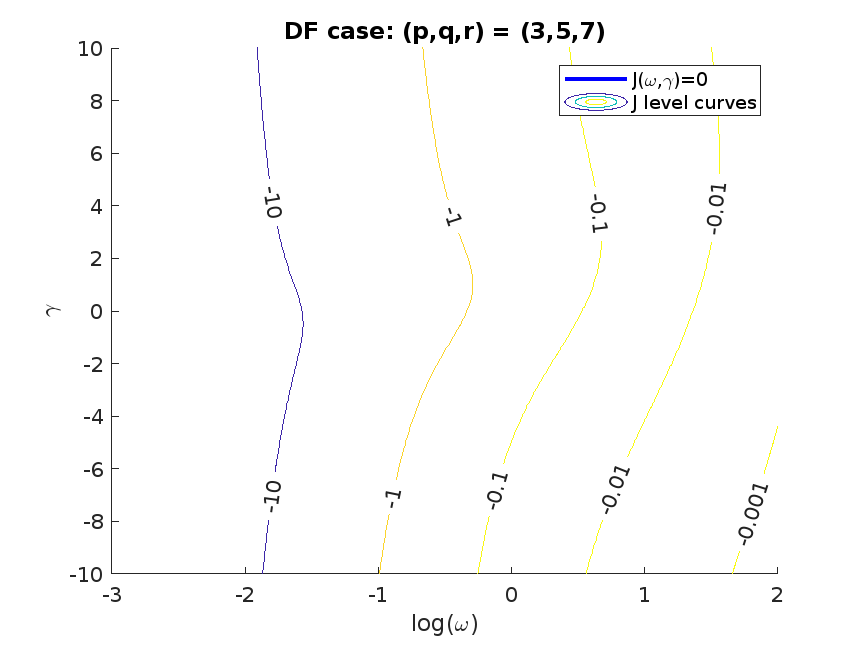}
\caption{DF case}\label{DFcase}
\end{figure}

\newpage
In the DD case, {$\omega^*(\gamma)$ in \eqref{1.4} exists for $-\infty<\gamma<\gamma_1$ and} $\Gamma_{\nex}$ meets the $\gamma$ axis at $\gamma_1$. 
The function $\omega^*(\gamma)$ decreases in $\gamma$, $\omega^*(\gamma_1)=0$ and $\omega^*(\gamma)\to \infty$ as $\gamma\to-\infty$, \crm (cf.~Proposition \ref{prop:Rex} part 4)\nc. 
For powers $2,3,7$, we have $2p+q\leq 7$, and it appears that $J(\omega,\gamma)>0$ for all $(\omega,\gamma)\in R_{\ex}$. For powers $2,4,7$, we have $2p+q>7$, but $p<\frac{7}{3}$. Thus, by Proposition \ref{prop:DDJ0limits}, $J(0,\gamma)<\infty$ for $\gamma<\gamma_1$, and {$\lim_{\gamma\to-\infty}J(0,\gamma)=0^-$}, $\lim_{\gamma\to \gamma_1^+}J(0,\gamma)=\infty$. Indeed, in the diagram for $2,4,7$, $J(\omega,\gamma)<0$ near the $\gamma$ axis for large $-\gamma$, and $J(\omega,\gamma)>0$ near the gamma axis for $\gamma$ close to $\gamma_1$. The zero level curve appears to have a limit in $(-\infty,\gamma_1)$ as $\omega\to 0$ in this case. For powers $3,4,7$ we have $J(\omega,\gamma)\to -\infty$ as $\omega\to 0$ for all $\gamma<\gamma_1$, \crm (cf.~Proposition \ref{prop:DDJ0limits})\nc. In the diagram for $3,4,7$, we see that $J(\omega,\gamma)<0$ near the $\gamma$ axis for all $\gamma<\gamma_1$, and the zero level curve appears to approach $(0,\gamma_1)$. Since $q<5$, we know by Proposition \ref{prop:DDlimits} \crm part 1 \nc that $J(\omega,\gamma)>0$ for large $-\gamma$, and indeed the zero level curve turns back towards the $\gamma$ axis in the diagram. For powers $3,5,7$, we have $q\geq 5$, and the zero level curve does not turn back towards the $\gamma$ axis, \crm which is in agreement with the theoretical result in Proposition \ref{prop:DDlimits} part 2.\nc

\begin{figure}[htp]
\includegraphics[scale = .5]{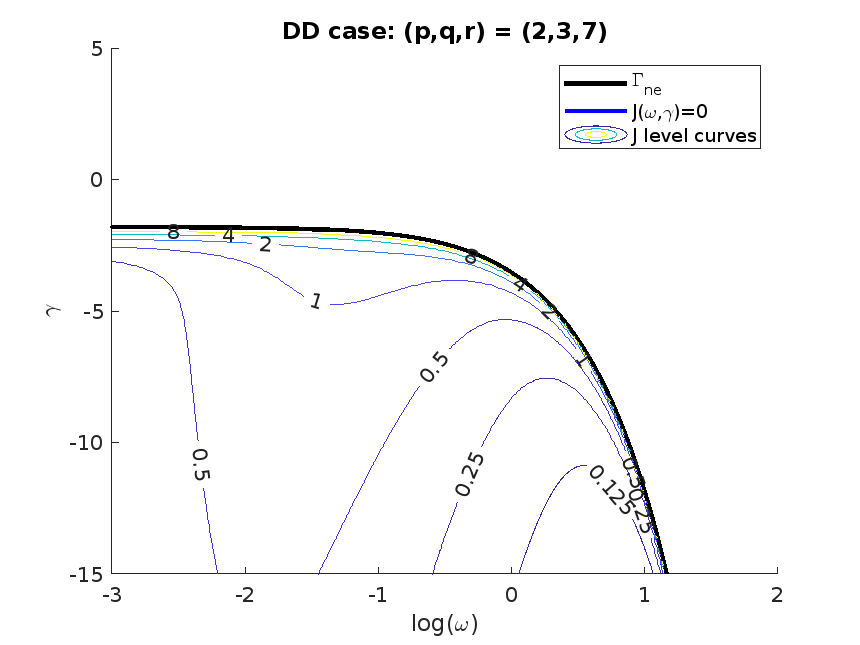}
\includegraphics[scale = .5]{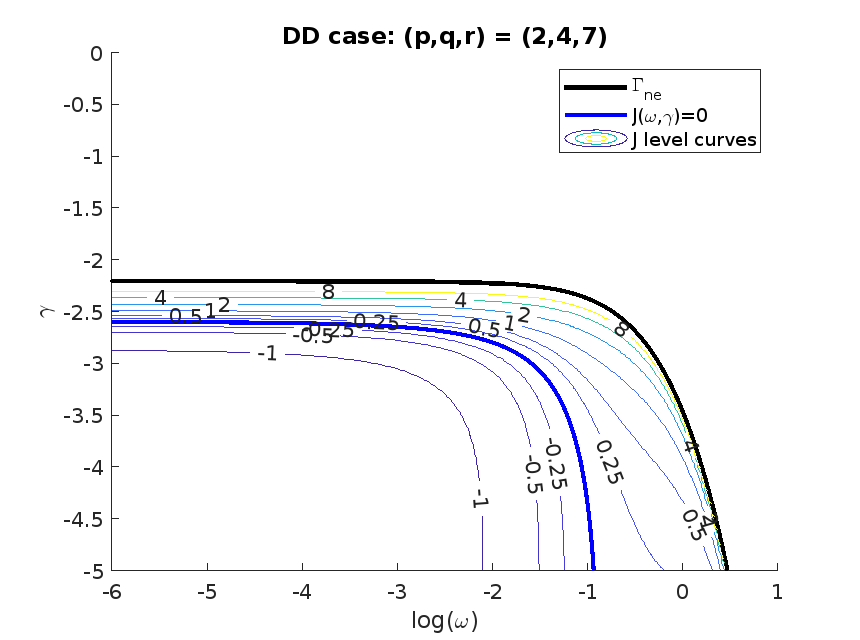}
\includegraphics[scale = .5]{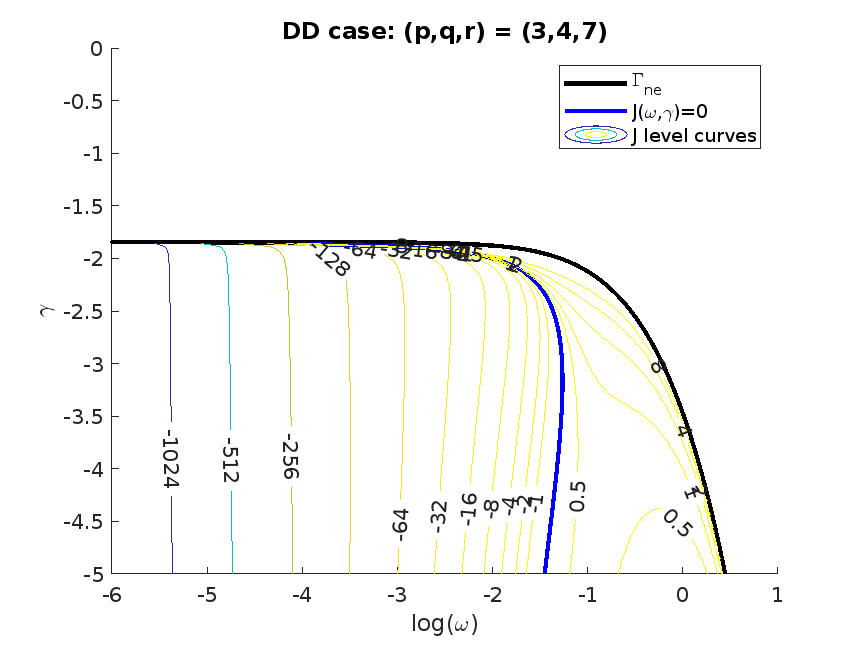}\hfill
\includegraphics[scale = .5]{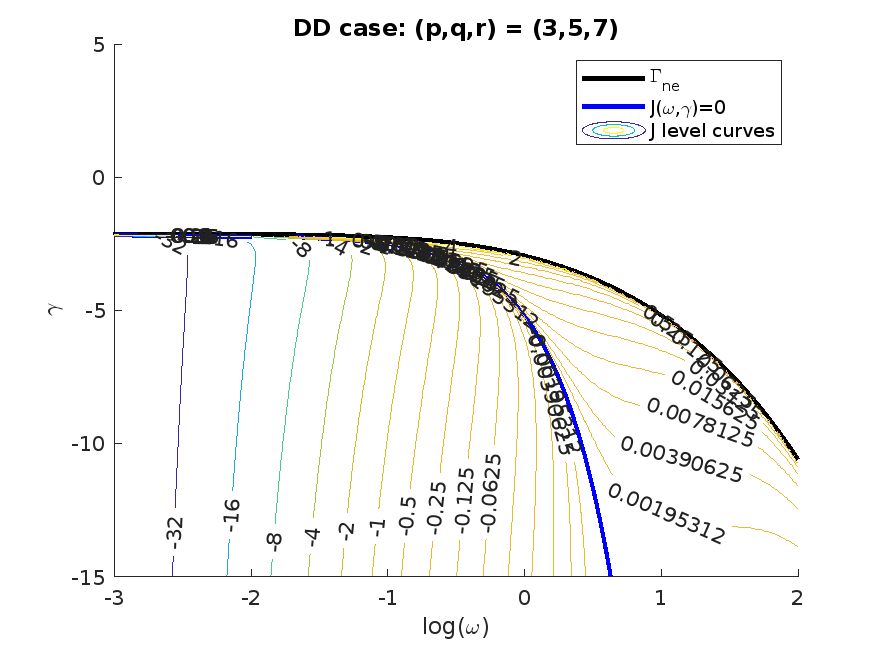}
\caption{DD case}\label{DDcase}
\end{figure}

\section{Preliminaries}\label{S3}
As explained in Section \ref{S1}, for the NLS \eqref{eq:nls}
we may consider $a_1=\pm1$, $a_2=-\gamma$, $a_3=\pm1$ for $\gamma\in\R$. Our standing wave profile $\phi$ then satisfies
\begin{align*}
\phi''=g(\phi) = \omega\phi -f(\phi),\quad f(\phi) = a_1|\phi|^{p-1}\phi-\gamma|\phi|^{q-1}\phi +a_3 |\phi|^{r-1}\phi,
\end{align*}
\begin{align*}
\quad \phi(0)>0,\quad \lim_{t\to\pm\infty}\phi(t) = 0.
\end{align*}
We use the following general existence result of Berestycki and Lions \cite{BeLi83-1} to determine the existence of solutions to this problem.
\begin{lemma}[\cite{BeLi83-1}] \label{BeLi}
Let $g\in C(\R)$ be a locally Lipschitz function with $g(0)=0$ and let $G(t)=\int_0^tg(s)ds$. A necessary and sufficient condition for the existence of a solution to the problem
\begin{align*}
\phi\in C^2(\R),\quad \lim_{t\to\pm\infty}\phi(t)=0,\quad\phi(0)>0,\quad\phi''=g(\phi),
\end{align*}
is that
\begin{equation}\label{existence}
\phi_0 = \inf\{t>0:G(t)=0\}\text{ exists},\quad \phi_0>0, \quad g(\phi_0)<0.
\end{equation}
When \eqref{existence} is satisfied, the solution $\phi(t)$ is unique up to translations, and this solution, after a suitable translation, is even, positive, $\phi(0)=\phi_0$, and $\phi'(t)<0$ for $t>0$.

\end{lemma}
Following \cite{IlKi93} and \cite[(2.8)]{LiTsZw21}, we define
\begin{equation}\label{U.def}
U(s) =U(\omega,\gamma,s)= 2G(\sqrt{s}) = \omega s -\frac{2a_1s^\frac{p+1}{2}}{p+1}+\frac{2\gamma s^\frac{q+1}{2}}{q+1}-\frac{2a_3s^\frac{r+1}{2}}{r+1}.
\end{equation}
We also define $F_1\in C(\R\times[0,\infty))$ by (it differs from \cite[(2.8)]{LiTsZw21} by a factor of $-s$)
\begin{align*}
sF_1(\gamma,s)= 2F(\sqrt{s})=\frac{2a_1s^\frac{p+1}{2}}{p+1}-\frac{2\gamma s^\frac{q+1}{2}}{q+1}+\frac{2a_3s^\frac{r+1}{2}}{r+1}
\end{align*}
(where $F(t)=\int_0^t f(s)ds$)
so that, for fixed $\gamma$,
\begin{align*}
U(s) = s(\omega-F_1(s)).
\end{align*}
The existence condition \eqref{existence} now reads (with $a=\phi_0^2$)
\begin{align}
a=\inf\{s>0: F_1(s)=\omega\}\text{ exists}, \quad a>0,\quad  U'(a)<0.
\end{align}

In the following 3 lemmas we describe the quantity $a$ as a function of $\omega$ and $\gamma$. Note that the existence of $a$ only implies $U'(a)\le0$, not $U'(a)<0$.

\begin{lemma}\label{lem:aomega}
Fix $\gamma\in \R$, and consider $a$ as a function of $\omega$. For any $\omega_1>0$, if $a(\omega_1)$ exists, then $a$ is defined for $\omega\in (0,\omega_1)$ and is increasing on $(0,\omega_1)$. In the F* cases, $a\to 0$ as $\omega\to 0$. In the *F cases, $a$ is defined for $\omega\in (0,\infty)$, and $a(\omega)\to\infty$ as $\omega\to \infty$. In the D* cases, there is an $a_0>0$ such that $a(\omega)>a_0$ for all $\omega>0$. In the *D cases, $U$ has no positive zeros for $\omega$ sufficiently large.
\end{lemma}
\begin{proof}
Let $\omega_2\in (0,\omega_1)$. Since $a(\omega_1)$ exists, $F_1(a(\omega_1))=\omega_1>\omega_2>0=F_1(0)$. By continuity of $F_1$, there is a $b\in (0, a(\omega_1))$ such that $F_1(b)=\omega_2$. Hence $a(\omega_2)$ exists and $a(\omega_2)\leq b<a(\omega_1)$.

In the F* cases, $F_1$ is increasing on a neighbourhood of $0$, so the first positive zero of $\omega-F_1(s)$ approaches $0$ as $\omega\to 0$.

In the *F cases, $F_1$ has a positive leading coefficient. As $F_1(0)=0$, this implies that $a(\omega)$ exists for all $\omega>0$. Since $F_1$ is continuous on $[0,\infty)$, we must have $a(\omega)\to \infty$ as $\omega\to \infty$.

In the D* cases, suppose that $\gamma$ is such that there is an $\omega_0>0$ such that $a(\omega_0,\gamma)$ exists. Since $a_1<0$, $F_1(s)<0$ for small $s>0$, and since $a(\omega_0)$ exists, $F_1$ has a smallest positive zero $a_0$. Hence $\omega-F_1(s)>\omega$ for $0<s<a_0$ and $\omega>0$, and hence $a(\omega)>a_0$ for all $\omega>0$.

In the *D cases, $F_1$ is bounded above on $(0,\infty)$. Hence $U$ has no positive zero for $\omega$ sufficiently large.
\end{proof}

\begin{lemma}\label{lem:agamma}
Fix $\omega>0$, and consider $a$ as a function of $\gamma$. For any $\gamma_1\in \R$, if $a(\gamma_1)$ exists, then $a(\gamma)$ exists for $\gamma<\gamma_1$ and is increasing on $(-\infty, \gamma_1)$. For any $\omega$, $a(\gamma)$ exists for $-\gamma$ sufficiently large. Moreover, $a(\gamma)\to 0$ as $\gamma\to -\infty$. In the *F cases, $a(\gamma)$ exists for all $\gamma\in \R$ and $a(\gamma)\to\infty$ as $\gamma\to \infty$.
\end{lemma}

\begin{proof}
Let $\gamma_2<\gamma_1$. For $\gamma = \gamma_1$, and $a=a(\gamma_1)$, we have
\begin{align}
\omega = F_1(a) = \frac{2a_1}{p+1}a^\frac{p-1}{2}-\frac{2\gamma}{q+1}a^\frac{q-1}{2}+\frac{2a_3}{r+1}a^\frac{r-1}{2},\label{F1}
\end{align}
and the right hand side is greater than $\omega$ for $\gamma = \gamma_2$, $a=a(\gamma_1)$. Hence, by continuity, there is a $b\in (0,a(\gamma_1))$ such that (\ref{F1}) is satisfied for $\gamma= \gamma_2$ and $a=b$. Hence $a(\gamma_2)$ exists, and $a(\gamma_2)\leq b< a(\gamma_1)$. For any fixed value of $a, \omega >0$, we can make the right hand side of (\ref{F1}) greater than $\omega$ by taking $-\gamma$ sufficiently large. It follows that $a(\gamma)$ exists for sufficiently large $-\gamma$, and $a(\gamma)\to 0$ as $\gamma\to -\infty$. In the *F cases, there is always an $a>0$ that satisfies (\ref{F1}). Moreover, if (\ref{F1}) can be satisfied for all $\gamma\in\R$, we must have $a\to \infty$ as $\gamma\to\infty$.
\end{proof}
\begin{lemma}\label{lem:altcts}
\textup{(a)} For any $\bar\omega>0$ and $\bar\gamma\in\R$ such that $a(\omega,\gamma)$ exists for $0<\omega<\bar\omega$, $\gamma<\bar\gamma$, $a(\bar\omega,\bar\gamma)$ exists and $\lim_{\omega\to \bar\omega^-,\gamma\to\bar\gamma^-}a(\omega,\gamma) = a(\bar\omega,\bar\gamma)$.

\textup{(b)}
For any $\bar \omega>0$ and $\bar \gamma\in\R$ such that $a(\omega,\gamma)$ exists for $\bar\omega < \omega<\bar\omega+\eps$, $\bar\gamma< \gamma<\bar\gamma+\eps$ for some $\eps>0$, the limit $\lim_{\omega\to \bar\omega^+,\gamma\to\bar\gamma^+}a(\omega,\gamma) $ and $a(\bar\omega,\bar\gamma)$ both exist. Denote  $U(\omega,\gamma,s) = U(s)$ for parameters $\omega,\gamma$ as in \eqref{U.def}. If there is a $\delta>0$ such that $U(\bar\omega,\bar\gamma,s)\geq 0$ for $a(\bar\omega,\bar\gamma)\leq s<a(\bar\omega,\bar\gamma)+\delta$, then $\lim_{\omega\to \bar\omega^+,\gamma\to\bar\gamma^+}a(\omega,\gamma)>a(\bar\omega,\bar\gamma)$ . Otherwise $\lim_{\omega\to \bar\omega^+,\gamma\to\bar\gamma^+}a(\omega,\gamma)=a(\bar\omega,\bar\gamma)$.
\end{lemma}

\begin{proof}
(a) Let $\omega_n$ and $\gamma_n$ be such that $\omega_n\nearrow\bar\omega$, $\gamma_n\nearrow\bar\gamma$. Let $b_n = a(\omega_n,\gamma_n)$. 
If $a_3>0$ so that $\lim_{s\to \infty} F_1(\gamma_n,s)=+\infty$, there is an $M>0$ such that $F_1(\gamma_n,s)>F_1(\bar\gamma,s)>\bar\omega>\omega_n$ for all $n\in\N$ and $s>M$.
If $a_3<0$ so that $\lim_{s\to \infty} F_1(\gamma_n,s)=-\infty$, there is an $M>0$ such that $F_1(\gamma_n,s)<F_1(\gamma_1,s)<\omega_1<\omega_n$ for all $n\in\N$ and $s>M$.
In either case, the choice of $M$ is uniform in $n$,
$b_n$ is bounded above by $M$ and increasing, so $b_n$ converges to some $b\leq M$. Since $F_1(\gamma,a)$ is continuous in $\gamma, a$, 
we have $\omega_n = F_1(\gamma_n,b_n)\to F_1(\bar\gamma,b)$ as $n\to\infty$. Hence $F_1(\bar\gamma,b) = \bar\omega$, so $a(\bar\omega,\bar\gamma)$ exists and $a(\bar\omega,\bar\gamma)\leq b$. As $a$ is increasing, $b_n\leq a(\bar\omega,\bar\gamma)$ for all $n\in\N$. Hence $a(\bar\omega,\bar\gamma) = b$. This also shows the limit $b$ is independent of the choice of sequence.
 
(b) Now suppose there is an $\eps>0$ such that $a(\omega,\gamma)$ exists for $\bar\omega<\omega<\bar\omega+\eps$, $\bar\gamma<\gamma<\bar\gamma+\eps$. By Lemma \ref{lem:aomega}, $a(\bar\omega,\bar\gamma)$ exists.

Suppose there is a $\delta>0$ such that $U(\bar\omega,\bar\gamma,s)\geq 0$ for $a(\bar\omega,\bar\gamma)\leq s< a(\bar\omega,\bar\gamma)+\delta$. As $U(\omega,\gamma,s)$ is strictly increasing in $\omega$ and $\gamma$ for all $s>0$, we then have $U(\omega,\gamma,s)>0$ for all $\omega>\bar\omega$, $\gamma>\bar\gamma$, and $0<s<a(\bar\omega,\bar\gamma)+\delta$. Therefore $\lim_{\omega\to \bar\omega^+,\gamma\to\bar\gamma^+}a(\omega,\gamma)\geq a(\bar\omega,\bar\gamma)+\delta$.

Otherwise, for any $b>a(\bar\omega,\bar\gamma)$ there is an $s_0\in (a(\bar\omega,\bar\gamma), b)$ such that $U(\bar\omega,\bar\gamma,b)<0$. For sufficiently large $n$, $U(\omega_n,\gamma_n, b)<0$ and so, by continuity, $U(\omega_n,\gamma_n, s)=0$ for some $0<s<b$. This shows that $a(\bar\omega,\bar\gamma)\leq \lim_{n\to\infty}a(\omega_n,\gamma_n)<b$ for all $b>a(\bar\omega,\bar\gamma)$, and so $\lim_{n\to\infty}a(\omega_n,\gamma_n)=a(\bar\omega,\bar\gamma)$.
\end{proof}

\emph{Remark.}\quad 
It is shown in \cite{LiTsZw21} for the nonlinearity $f(u) =  |u|u -\gamma |u|^2u + |u|^3u$ (FF case) that $a(\omega,\gamma)$ is defined for every $\omega>0$ and $\gamma\in\R$. It is continuous on $\omega,\gamma$ except on the nonexistence curve $\Gamma_{\nex}$. 
 As $(\omega,\gamma)\to (\omega_0,\gamma_0)\in\Gamma_{\nex}$, the value of $a(\omega,\gamma)$
converges to $a(\omega_0,\gamma_0)$ from the left lower side of $\Gamma_{\nex}$, and converges to another value $b\ge a(\omega_0,\gamma_0)$ from the right upper side. 
The limit $b$ agrees with $a(\omega_0,\gamma_0)$ if $(\omega_0,\gamma_0)$ is the endpoint of $\Gamma_{\nex}$, and is strictly larger otherwise.
Lemma \ref{lem:altcts} shows 
this is also true for the FF case of general triple power nonlinearity considered in this paper.

\medskip

Consider
a family of even standing waves $\phi_\omega$, $\omega\in (\omega_0,\omega_1)$, of \eqref{eq:nls} for general $f(u)$. 
We can talk about $\phi_\omega$ because its uniqueness for fixed $\omega$ is given by Lemma \ref{BeLi}.
Iliev-Kirchev \cite{IlKi93} gave a stability criterion in terms the mass functional $Q(\phi_\omega)$, where   
\begin{align*}
Q(u) = \int_\R |u|^2dx.
\end{align*}

\begin{theorem}[Iliev-Kirchev \cite{IlKi93}]\label{thm:stab}
Suppose $f(u)$ is such that \eqref{eq:nls} is locally wellposed in the Sobolev space $H^2(\R)$, and there is a constant $A>0$ such that $U'(s)\in C^0[0,A)\cap C^1(0,A)$, $sU''(s)\to 0$ as $s\to 0$ and the existence condition \eqref{existence} is satisfied with $a<A$. If $\frac{\partial}{\partial\omega}Q(\phi_\omega)>0$, then the standing wave $e^{i\omega t} \phi_\omega(x)$ is stable. If $\frac{\partial}{\partial\omega}Q(\phi_\omega)<0$, then the standing wave $e^{i\omega t} \phi_\omega(x)$ is unstable. Moreover,
\begin{align}
\frac{\partial}{\partial \omega} Q(\phi_\omega) = \frac{-1}{2U'(a)}\int_0^a\left(3+\frac{s(U'(a)-U'(s))}{U(s)}\right)\frac{\sqrt{s}}{\sqrt{U(s)}}\,ds.\label{I-KJ}
\end{align}
\end{theorem}
The above formula is \cite[(2.11)]{LiTsZw21} and is equivalent to that in \cite[Lemma 6]{IlKi93}.

For convenience, we define
\begin{align}
J(\omega,\gamma) = \frac{\partial}{\partial \omega} Q(\phi_\omega).
\end{align}
In our case, $U(s)/s = \omega - \frac{2a_1}{p+1}s^\frac{p-1}{2} + \frac{2\gamma}{q+1}s^\frac{q-1}{2} - \frac{2a_3}{r+1}s^\frac{r-1}{2}$, and subtracting $U(a)/a=0$ gives
\begin{align*}
\frac{U(s)}{s} = \frac{2a_1}{p+1}(a^\frac{p-1}{2}-s^\frac{p-1}{2}) - \frac{2\gamma}{q+1}(a^\frac{q-1}{2}-s^\frac{q-1}{2}) + \frac{2a_3}{r+1}(a^\frac{r-1}{2}-s^\frac{r-1}{2}).
\end{align*}
We also have
\begin{align*}
U'(a)-U'(s)=- a_1(a^\frac{p-1}{2}-s^\frac{p-1}{2})+ \gamma (a^\frac{q-1}{2}-s^\frac{q-1}{2})- a_3(a^\frac{r-1}{2}-s^\frac{r-1}{2}).
\end{align*}
Thus (\ref{I-KJ}) becomes
\begin{align*}
\frac{-1}{2U'(a)}\int_0^a\frac{\frac{a_1(5-p)}{p+1}(a^\frac{p-1}{2}-s^\frac{p-1}{2}) - \frac{\gamma(5-q)}{q+1}(a^\frac{q-1}{2}-s^\frac{q-1}{2}) + \frac{a_3(5-r)}{r+1}(a^\frac{r-1}{2}-s^\frac{r-1}{2})}{ \left(\frac{2a_1}{p+1}(a^\frac{p-1}{2}-s^\frac{p-1}{2}) - \frac{2\gamma}{q+1}(a^\frac{q-1}{2}-s^\frac{q-1}{2}) + \frac{2a_3}{r+1}(a^\frac{r-1}{2}-s^\frac{r-1}{2})\right)^{3/2}}\,ds.
\end{align*}
Note that the denominator of the integrand is $(U(s)/s)^{3/2}$, and is therefore positive for $s\in [0,a)$. We now use a change of variables to integrate over a constant interval, and get the following lemma.
\begin{lemma}\label{lem:pqrstab}
For the particular choice $f(x) = a_1x|x|^{p-1} - \gamma x|x|^{q-1} + a_3 x|x|^{r-1}$, we have
\begin{align}
J(\omega, \gamma)
 &=C \int_0^1 \frac{\frac{a_1(5-p)}{p+1}(1-s^\frac{p-1}{2})a^\frac{p-1}{2}-\frac{\gamma(5-q)}{q+1}(1-s^\frac{q-1}{2})a^\frac{q-1}{2}+\frac{a_3(5-r)}{r+1}(1-s^\frac{r-1}{2})a^\frac{r-1}{2}}{\left[\frac{a_1}{p+1}(1-s^\frac{p-1}{2})a^\frac{p-1}{2}-\frac{\gamma}{q+1}(1-s^\frac{q-1}{2})a^\frac{q-1}{2}+\frac{a_3}{r+1}(1-s^\frac{r-1}{2})a^\frac{r-1}{2}\right]^\frac{3}{2}}ds\nonumber\\
&=C \int_0^1 \frac{a_1(5-p)A_p(a,s)-\gamma(5-q)A_q(a,s)+a_3(5-r)A_r(a,s)}{\left[a_1A_p(a,s)-\gamma A_q(a,s)+a_3A_r(a,s)\right]^\frac{3}{2}}ds,\label{pqrstab}
\end{align}
where $C=C(\omega,\gamma) = \frac{-a}{4\sqrt{2}U'(a)}$, $A_l(a,s) = \frac{1-s^\frac{l-1}{2}}{l+1}a^\frac{l-1}{2}$ for $l=p,q,r$, and the denominator is positive for $s\in [0,1)$.
\end{lemma}
We also write
\begin{equation}\label{ND.def}
\begin{split}
    N(a,s) & = a_1(5-p)A_p(a,s)-\gamma(5-q)A_q(a,s)+a_3(5-r)A_r(a,s),
\\[2pt]
    D(a,s) &= a_1A_p(a,s)-\gamma A_q(a,s)+a_3A_r(a,s).
\end{split}    
\end{equation}
Since $D(a,s)>0$ for all $s\in [0,1)$, we can show that $J(\omega,\gamma)>0$ by approximating $N(a,s)$ well enough to show that $N(a,s)>0$ for all $s\in (0,1)$.  For the purpose of approximations, it is useful to note that $A_l(a,s)/(1-s)$ and $(1-s)/A_l(a,s)$ are both $L^\infty([0,1])$ as functions of $s$. This is implied by the following lemma,  which is also used in the proof of Propositions \ref{prop:DFJ0pos} and \ref{prop:DFJ0neg}.
\begin{lemma}\label{lem:lineq}
Let $h(x)=\frac{x^{p_1}-x^{q_1}}{x^{p_2}-x^{q_2}}$ for some $q_1>p_1\ge 0$ and $q_2>p_2\ge 0$. If $p_1\ge p_2$ and $q_1>q_2$, then $h'(x)>0$ for all $x\in (0,1)$ and $h(x)\leq \frac{q_1-p_1}{q_2-p_2}$. If $p_1\le p_2$ and $q_1<q_2$, then $h'(x)<0$ for all $x\in(0,1)$ and $h(x)\geq \frac{q_1-p_1}{q_2-p_2}$. 
\end{lemma}
\begin{proof}
    Suppose $p_1 > p_2$ and $q_1>q_2$. We have
\begin{align}
        h'(x) &= \frac{(p_1-p_2)x^{p_1+p_2}+(q_2-p_1)x^{p_1+q_2}+(p_2-q_1)x^{p_2+q_1}+(q_1-q_2)x^{q_1+q_2}}{x(x^{p_2}-x^{q_2})^2}.\label{h'}
    \end{align}

    If $q_2\ge p_1$, then
    \begin{align*}
        \lambda_1(p_1+p_2) + \lambda_2(p_1+q_2)+\lambda_3(q_1+q_2)=p_2+q_1,
    \end{align*}
    where
    \begin{align*}
        \lambda_1 = \frac{p_1-p_2}{q_1-p_2},\quad \lambda_2 = \frac{q_2-p_1}{q_1-p_2},\quad \lambda_3 = \frac{q_1-q_2}{q_1-p_2},\quad \lambda_1+\lambda_2+\lambda_3 = 1.
    \end{align*}
    Thus, by convexity of $s\mapsto x^s$,
    \begin{align}\label{3.9}
        \lambda_1 x^{p_1+p_2}+\lambda_2 x^{p_1+q_2}+\lambda_3x^{q_1+q_2}> x^{p_2+q_1}.
    \end{align}
    Multiplying by $q_1-p_2$ shows that the numerator of (\ref{h'}) is positive.

The above argument remains correct if $p_2=p_1<q_2$, and the inequality in \eqref{3.9} stays strict. The case $p_2=p_1=q_2$ is excluded by $p_2<q_2$.

    Now suppose $q_2<p_1$. For 
    \begin{align*}
        \lambda_1 = \frac{p_1-p_2}{p_1-p_2+q_1-q_2},\quad \lambda_2 = \frac{q_1-q_2}{p_1-p_2+q_1-q_2},
    \end{align*}
    \begin{align*}
             \lambda_3 = \frac{p_1-q_2}{p_1-p_2+q_1-q_2},\quad \lambda_4 = \frac{q_1-p_2}{p_1-p_2+q_1-q_2},
    \end{align*}
    we have $\lambda_1+\lambda_2=\lambda_3+\lambda_4 = 1$, and
    \begin{align*}
    \lambda_1(p_1+p_2)+\lambda_2(q_1+q_2) = \lambda_3(p_1+q_2) + \lambda_4 (p_2+q_1).
    \end{align*}
    By convexity of $s\mapsto x^s$ and that 
    $p_1+q_2$ and $p_2+q_1$ are between $p_1+p_2$ and $q_1+q_2$, it follows that
    \begin{align}
        \lambda_1 x^{p_1+p_2} + \lambda_2 x^{q_1+q_2}>\lambda_3 x^{p_1+q_2} + \lambda_4 x^{q_1+p_2}.
    \end{align}
    Multiplying by $p_1-p_2+q_1-q_2$ shows that the numerator of (\ref{h'}) is positive.
    
By L'Hopital's rule, $\lim_{x\to 1} h(x)=\frac{q_1-p_1}{q_2-p_2}$, so $h(x)\leq \frac{q_1-p_1}{q_2-p_2}$ for $x\in (0,1]$.

Since $h(x)>0$ for $x\in (0,1)$, the case for $p_1<p_2$ and $q_1<q_2$ follows by taking inverses.
\end{proof}

Following Kfoury-Le Coz-Tsai in \cite{MR4480890}, we use the beta function to calculate the limits of $J(\omega,\gamma)$ as $\omega\to0^+$ in the D* cases. Recall that the beta function is defined for $x,y>0$ by
\begin{align*}
B(x,y)=\int_0^1t^{x-1}(1-t)^{y-1}dt.
\end{align*}
The beta function is related to the gamma function by 
\begin{align*}
B(x,y)=\frac{\Gamma(x)\Gamma(y)}{\Gamma(x+y)}.
\end{align*}
We also define the function $H$ for $x,y>0$ by
\begin{align*}
H(x,y)=\int_0^1\frac{t^{x-1}(1-t^y)}{(1-t)^\frac{3}{2}}dt.
\end{align*}
The following lemma is \cite[Lemma 9]{MR4480890} and describes the relation between the functions $H$ and $B$.
\begin{lemma}[Kfoury-Le Coz-Tsai \cite{MR4480890}]\label{lem:BvH}
For $x,y>0$, we have
\begin{align*}
H(x,y) = -(2x-1)B(x,1/2)+(2x+2y-1)B(x+y,1/2).
\end{align*}
\end{lemma}
Using this we can calculate the following integral. It is \cite[Lemma 10]{MR4480890} except the explicit constant and change of variables. 
We skip its calculation.
\begin{lemma}\label{lem:2pow}
For any $1<p<q$ with $p<\frac{7}{3}$, we have
\begin{align*}
\int_0^1\frac{-(5-p)(1-s^\frac{p-1}{2})+(5-q)(1-s^\frac{q-1}{2})}{(s^\frac{p-1}{2}-s^\frac{q-1}{2})^\frac{3}{2}}ds=2\frac{7-2p-q}{q-p}B\left(\frac{7-3p}{2(q-p)},\frac{1}{2}\right).
\end{align*}
\end{lemma}

The following lemma, which is well-known for integer powers as Descartes' rule of signs, is given for real powers in Haukkanen-Tossavainen \cite[Theorem 2.2]{MR2831628}.  
\begin{lemma}\label{lem:RoS}
Let $p_1<p_2<\dots<p_n\in \R$ for some $n\in\N$, and let $c_1,c_2,\dots,c_n\in\R\setminus\{0\}$. Define $f:[0,\infty)\to \R$ by
\[
f(s)=\sum_{i=1}^n c_is^{p_i}.
\]
Then the number of positive real zeros of $f$ is at most $|\{i:c_ic_{i+1}<0\}|$, the number of sign changes of the coefficients $c_i$.
\end{lemma}

\section{Existence for triple power nonlinearities}\label{S4}
In this section we study 
the set $R_{\ex}$ of $(\omega,\gamma)$ for which a standing wave solution exists, and its boundary $\Gamma_{\nex}$, for each of the 4 cases FF, FD, DF, and DD.

\begin{lemma}\label{lem:preEx1}
In any case of FF, FD, DF, and FD, $R_{\ex}$ is open, and
$\Gamma_{\nex}$ is the set of $(\omega,\gamma)$ such that $a(\omega,\gamma)$ exists and $U'(a(\omega,\gamma))=0$.
\end{lemma}
\begin{proof}
Suppose $(\omega_0,\gamma_0)\in R_{\ex}$, i.e., 
$a(\omega_0,\gamma_0)>0$ exists and $U'(a(\omega_0,\gamma_0))<0$. 
Then the implicit function theorem would show that $a(\omega,\gamma)$ exists and is a continuously differentiable function of $\omega$ and $\gamma$ on a neighbourhood of $(\omega_0,\gamma_0)$. By continuity of $U'(s)$ as a function of $s$, $\omega$, and $\gamma$, it would follow that $U'(a(\omega,\gamma))<0$ for $(\omega,\gamma)$ in a neighbourhood of $(\omega_0,\gamma_0)$. Hence $R_{\ex}$ is open.

Suppose $a(\omega_0,\gamma_0)$ exists and $U'(a(\omega_0,\gamma_0))=0$. Then by Lemma \ref{lem:aomega}, $a(\omega,\gamma_0)$ exists for all $0<\omega<\omega_0$. Differentiating $U(s) = \omega s - sF_1(s)$ gives
\begin{align*}
U'(a(\omega,\gamma_0)) = \omega - F_1(a(\omega,\gamma_0)) - a(\omega,\gamma_0)F_1'(a(\omega,\gamma_0))= -a(\omega,\gamma_0)F_1'(a(\omega,\gamma_0)).
\end{align*}
As $F_1(a)$ is a sum of finitely many powers of $a$, there are finitely many $a>0$ such that $aF_1'(a) = 0$. Since $a(\omega,\gamma_0)$ is increasing in $\omega$, it follows that there are finitely many $0<\omega<\omega_0$ such that $U'(a(\omega,\gamma_0))=0$. Thus, there are $\omega$ arbitrarily close to $\omega_0$ such that $(\omega,\gamma_0)$ satisfy the existence criterion (\ref{existence}). Since $(\omega_0,\gamma_0)\notin R_{\ex}$ this shows that $(\omega_0,\gamma_0)\in \Gamma_{\nex}$.

Conversely, suppose $(\omega_0,\gamma_0)\in \Gamma_{\nex}$. Then $a(\omega,\gamma)$ exists for some $(\omega,\gamma)$ arbitrarily close to $(\omega_0,\gamma_0)$. By Lemmas \ref{lem:aomega} and \ref{lem:agamma}, it follows that $a(\omega,\gamma)$ exists for all $\omega<\omega_0$ and $\gamma<\gamma_0$. Thus, by Lemma \ref{lem:altcts}, $a(\omega_0,\gamma_0)$ exists. If we had $U'(a(\omega_0,\gamma_0))<0$, then $(\omega_0,\gamma_0)\in R_{\ex}$.
This contradicts $(\omega_0,\gamma_0)\in\Gamma_{\nex}$ as $R_{\ex}$ is open.
So we must have $U'(a(\omega_0,\gamma_0))=0$.
\end{proof}
\begin{lemma}\label{lem:preEx2}
For $\omega>0$, $\gamma\in\R$, if $b>0$ is such that $U(b)=U'(b)=0$, then $U''(b)\geq 0$ if and only if $b=a(\omega,\gamma)$. 
\end{lemma}
\begin{proof}
Let $b>0$ be such that $U(b)=U'(b)=0$. If $U''(b)<0$, then, since $U'(0)=\omega>0$, $U$ has a zero in $(0,b)$. Hence $b=a(\omega,\gamma)$ implies $U''(b)\ge0$.
Conversely, suppose there is a $c\in (0,b)$ with $U(c)=0$. If $U''(b)>0$, then $U$ has positive local maxima at some $c_1\in (0,c)$ and $c_3\in (c,b)$. Then $U$ also has a local minimum $c_2\in (c_1,c_3)$, so $U'$ has at least four positive zeros $c_1,c_2,c_3,b$. As $U'$ is a sum of four powers, this contradicts Lemma \ref{lem:RoS}. If $U''(b)=0$, then $U'$ has at least two zeros $c_1\in (0,c)$, $c_2\in (c,b)$, and $U''$ has at least three zeros $d_1\in (c_1,c_2)$, $d_2\in (c_2,b)$ and $b$. Since $U''$ is a sum of three powers, this contradicts Lemma \ref{lem:RoS}. Hence when $U''(b)\ge0$, such $c$ does not exist, and $b=a(\omega,\gamma)$.
\end{proof}
\begin{proposition}\label{prop:Rex}
For $\omega>0$, $\gamma\in\R$, we have $(\omega,\gamma)\in \Gamma_{\nex}$ if and only if $(\omega,\gamma)=(\omega_{\nex}(a),\gamma_{\nex}(a))$ and $U''(\omega,\gamma,a)\geq 0$ for some $a>0$, where
\begin{align*}
\omega_{\nex}(a) &= \frac{2a_1(q-p)}{(q-1)(p+1)}a^\frac{p-1}{2}-\frac{2a_3(r-q)}{(q-1)(r+1)}a^\frac{r-1}{2} ,
\\
\gamma_{\nex}(a)&=\frac{q+1}{q-1}\left(a_1\frac{p-1}{p+1}a^\frac{p-q}{2}+a_3\frac{r-1}{r+1}a^\frac{r-q}{2}\right).
\end{align*}
As a consequence, for each $\gamma\in \R$, there is at most one value $\omega^*(\gamma)>0$ such that $(\omega^*(\gamma),\gamma)\in\Gamma_{\nex}$. The existence regions and $\Gamma_{\nex}$ in each case are as follows:
\begin{enumerate}
\item FF case: $\Gamma_{\nex}$ is parameterized by $(\omega_{\nex}(a),\gamma_{\nex}(a))$ for $0<a\le \aFF$ where $\aFF^\frac{r-p}{2}=\frac{(q-p)(p-1)(r+1)}{(r-q)(r-1)(p+1)}$, 
or by $(\omega^*(\gamma),\gamma)$ for $\gamma\ge \gamma_1=\gamma_{\nex}(\aFF)$,
and $R_{\ex}$ is the complement of $\Gamma_{\nex}$.
\item FD case: $\Gamma_{\nex}$ is parameterized by $(\omega_{\nex}(a),\gamma_{\nex}(a))$ for $a>0$. The existence region is $\{(\omega,\gamma): 0<\omega<\omega^*(\gamma), \gamma\in\R\}$.
\item DF case: Solutions exist for all $\omega>0$ and $\gamma\in\R$.
\item DD case: $\Gamma_{\nex}$ is parameterized by $(\omega_{\nex}(a),\gamma_{\nex}(a))$ for $a>\aDD$ where $\aDD^\frac{r-p}{2}=\frac{(q-p)(r+1)}{(r-q)(p+1)}$. Noting $\omega_{\nex}(\aDD)=0$ and letting $\gamma_1 = \gamma_{\nex}(\aDD)$, we have $R_{\ex}=\{(\omega,\gamma): 0<\omega<\omega^*(\gamma), \gamma<\gamma_1\}$.\end{enumerate}
\end{proposition}

\begin{proof}
By Lemmas \ref{lem:preEx1} and \ref{lem:preEx2}, $(\omega,\gamma)\in\Gamma_{\nex}$ if and only if there is an $a>0$ such that 
\begin{align*}
\frac{U(a)}{a}&= \omega -\frac{2a_1}{p+1}a^\frac{p-1}{2}+\frac{2\gamma}{q+1}a^\frac{q-1}{2}-\frac{2a_3}{r+1}a^\frac{r-1}{2}=0,\\
U'(a) &= \omega -  a_1a^\frac{p-1}{2} + \gamma a^\frac{q-1}{2} - a_3a^\frac{r-1}{2} = 0,
\end{align*}
and $U''(a)\geq 0$. Subtracting to eliminate $\omega$ yields
\begin{align*}
&\frac{a_1(p-1)}{p+1}a^\frac{p-1}{2} - \frac{\gamma(q-1)}{q+1}a^\frac{q-1}{2} + \frac{a_3(r-1)}{r+1}a^\frac{r-1}{2}=0\\
\iff&\gamma = a_1\frac{(p-1)(q+1)}{(q-1)(p+1)}a^\frac{p-q}{2}+ a_3\frac{(r-1)(q+1)}{(q-1)(r+1)}a^\frac{r-q}{2},
\end{align*}
and substituting to solve for $\omega$ yields
\begin{align*}
\omega = a_1\frac{2(q-p)}{(q-1)(p+1)}a^\frac{p-1}{2}-a_3\frac{2(r-q)}{(q-1)(r+1)}a^\frac{r-1}{2}.
\end{align*}
$\Gamma_{\nex}$ is therefore parameterized by $(\omega_{\nex}(a),\gamma_{\nex}(a))$ for $a$ such that  $\omega>0$ and $U''(a)\geq 0$.
The condition $\omega>0$ amounts to 
\begin{equation}\label{omega>0}
a_1\frac{(q-p)(r+1)}{(r-q)(p+1)}>a_3a^\frac{r-p}{2}.
\end{equation}
For $U''(a)\geq 0$, substituting for $\gamma_{\nex}(a)$ in $U''(a)$ gives, 
\begin{align}
\nonumber
&U''(a) = \left(-\frac{p-1}{2} + \frac{(p-1)(q+1)}{2(p+1)}\right)a_1 a^\frac{p-3}{2}+ \left( \frac{(r-1)(q+1)}{2(r+1)} - \frac{r-1}{2}\right)a_3 a^\frac{r-3}{2}\geq 0\\
&\iff a_3 a^\frac{r-p}{2}\leq a_1\frac{(q-p)(p-1)(r+1)}{(r-q)(r-1)(p+1)}.
\label{U''ge0}
\end{align}

We now consider the 4 cases.

\smallskip
\noindent\textbf{FF case:}
For $\omega>0$ and $U''(a)\geq 0$, by \eqref{omega>0} and \eqref{U''ge0}, 
\[
a^\frac{r-p}{2}<\frac{(q-p)(r+1)}{(r-q)(p+1)},\quad a^\frac{r-p}{2}\le \frac{(q-p)(p-1)(r+1)}{(r-q)(r-1)(p+1)}.
\]
Since the second bound is smaller, the first condition is redundant. Hence $\Gamma_{\nex}$ is parameterized by $(\omega_{\nex}(a),\gamma_{\nex}(a))$ for  $0<a\le \aFF$, where $\aFF^\frac{r-p}{2}=\frac{(q-p)(p-1)(r+1)}{(r-q)(r-1)(p+1)}$. When $a\le \aFF$, a calculation shows that $\gamma_{\nex}'(a)\leq0$ and $\omega_{\nex}'(a)\geq0$, so $\omega^*(\gamma)$ is well-defined and decreasing for $\gamma\geq \gamma_1:= \gamma_{\nex}(\aFF)$. Since $a(\omega,\gamma)$ exists for all $(\omega,\gamma)$ in the FF case by Lemma \ref{lem:aomega}, and $U'(a(\omega,\gamma))
\neq 0$ for $(\omega,\gamma)\notin \Gamma_{\nex}$, the existence criteria are satisfied for all $(\omega,\gamma)\notin \Gamma_{\nex}$.

\smallskip
\noindent\textbf{FD case:}
For $\omega>0$ and $U''(a)\geq 0$, by \eqref{omega>0} and \eqref{U''ge0}, they are satisfied 
for all $a>0$.  Hence $\Gamma_{\nex}$ is parameterized by $(\omega_{\nex}(a),\gamma_{\nex}(a))$ for $a>0$. A calculation shows that $\omega_{\nex}'(a)>0$ and $\gamma_{\nex}'(a)<0$, so the function $\omega^*(\gamma)$ is well-defined and decreasing for $\gamma\in \R$. For any $\gamma\in\R$, if $\omega^+(\gamma) = \sup\{\omega:(\omega,\gamma)\in R_{\ex}\}$ such that $\omega^+(\gamma)<\infty$, then $(\omega^+,\gamma)\in \Gamma_{\nex}$. By Lemma \ref{lem:aomega}, $\omega^+(\gamma)<\infty$, so $\omega^+(\gamma)=\omega^*(\gamma)$. Hence $(\omega,\gamma)\in R_{\ex}$ if and only if $\omega<\omega^*(\gamma)$.

\smallskip
\noindent\textbf{DF case:}
We have $\omega_{\nex}(a)<0$ for all $a>0$, so $\Gamma_{\nex}=\emptyset$. Since $a(\omega,\gamma)>0$ for all $\omega>0$, $\gamma\in\R$, the existence criteria are satisfied on the entire half plane.

\smallskip
\noindent\textbf{DD case:}
For $\omega>0$ and $U''(a)\geq 0$, by \eqref{omega>0} and \eqref{U''ge0}, 
\[
a^\frac{r-p}{2}>\frac{(q-p)(r+1)}{(r-q)(p+1)},\quad a^\frac{r-p}{2}\ge \frac{(q-p)(p-1)(r+1)}{(r-q)(r-1)(p+1)}.
\]
Since the second bound is smaller, the second condition is redundant.
Hence $\Gamma_{\nex}$ is parameterized by $(\omega_{\nex}(a),\gamma_{\nex}(a))$ for $a>\aDD$ where $\aDD^\frac{r-p}{2}=\frac{(q-p)(r+1)}{(r-q)(p+1)}$. For $a>\aDD$,
 a calculation shows that $\omega_{\nex}'(a)>0$ and $\gamma_{\nex}'(a)<0$. Hence $\omega^*(\gamma)$ is well defined and decreasing for $\gamma<\gamma_1:=\gamma_{\nex}(\aDD)$. Suppose $(\omega_0,\gamma_0)\in R_{\ex}$. By Lemma \ref{lem:aomega}, $\omega^+(\gamma_0) = \sup\{\omega:(\omega,\gamma)\in R_{\ex}\}<\infty$, so $(\omega^+(\gamma),\gamma)=(\omega^*(\gamma),\gamma)\in \Gamma_{\nex}$. Since $\gamma_{\nex}$ is decreasing, we then have $\gamma<\gamma_1$. Conversely, if $\gamma<\gamma_1$ and $\omega_0<\omega^*(\gamma)$, then, by Lemma \ref{lem:aomega}, $a(\omega_0,\gamma)$ exists. Since $(\omega_0,\gamma)\notin\Gamma_{\nex}$, we also have $U'(a(\omega_0,\gamma))\neq 0$. Thus, the existence criteria are satisfied for $(\omega,\gamma)$. Note that $\omega_{\nex}(\aDD)=0$ and there is no solution for $ \gamma \ge \gamma_1$.
\end{proof}

\section{Limits of the stability functional near the nonexistence curve} 
\label{S5}
The following proposition generalizes Proposition 4.1 in \cite{LiTsZw21} to arbitrary $1<p<q<r$.
\begin{proposition}\label{prop:Gamlimits}
Let $(\omega_0,\gamma_0)$ be a point on $\Gamma_{\nex}$ that is not an endpoint of the parameterization in Proposition \ref{prop:Rex}. Then $\lim_{\omega\to \omega_0^-\gamma\to\gamma_0^-} J(\omega,\gamma)=+\infty$ and in the FF case $\lim_{\omega\to \omega_0^+,\gamma\to\gamma_0^+}J(\omega,\gamma)=-\infty$.
\end{proposition}

\begin{proof}
Recall we denote the integrand of $J(\omega,\gamma)$ in \eqref{pqrstab} as $N(a,s)/D(a,s)^{3/2}$.
We first consider $N(a,s)$ for $s$ close to $1$. By Lemma \ref{lem:altcts} $a(\omega,\gamma)\to a_0=a(\omega_0,\gamma_0)$ as $\omega\nearrow\omega_0,\gamma\nearrow \gamma_0$, and by Proposition \ref{prop:Rex}, $(\omega_0,\gamma_0)=(\omega_{\nex}(a_0),\gamma_{\nex}(a_0))$. Thus, as $\omega\nearrow\omega_0,\gamma\nearrow \gamma_0$,
\begin{align}\nonumber
\frac{N(a(\omega,\gamma),s)}{1-s^\frac{p-1}{2}}\to\ &
\frac{N(a_0,s)}{1-s^\frac{p-1}{2}} 
=a_1\frac{5-p}{p+1}a_0^\frac{p-1}{2}-a_1\frac{(5-q)(p-1)(1-s^\frac{q-1}{2})}{(q-1)(p+1)(1-s^\frac{p-1}{2})}a_0^\frac{p-1}{2}\\
&-a_3\frac{(5-q)(r-1)(1-s^\frac{q-1}{2})}{(q-1)(r+1)(1-s^\frac{p-1}{2})}a_0^\frac{r-1}{2}+a_3\frac{(5-r)(1-s^\frac{r-1}{2})}{(r+1)(1-s^\frac{p-1}{2})}a_0^\frac{r-1}{2}\label{eq5.1}
\end{align}
and using Lemma \ref{lem:lineq} applied to $\frac{1-s^\frac{l-1}{2}}{1-s^\frac{p-1}{2}}$ for $l=q,r$, we see that this convergence is uniform on $[0,1]$. Moreover, $\frac{1-s^\frac{l-1}{2}}{1-s^\frac{p-1}{2}}\to \frac{l-1}{p-1}$ as $s\to 1$, so
\begin{align*}
\frac{N(a_0,s)}{1-s^\frac{p-1}{2}}
\to \frac{a_1(q-p)}{p+1}a_0^\frac{p-1}{2}+\frac{a_3(r-1)(q-r)}{(p-1)(r+1)}a_0^\frac{r-1}{2}\quad\text{as } s\to 1.
\end{align*}
The bounds on $a$ given in Proposition \ref{prop:Rex} ensure that this quantity is positive in the FD and DD cases, and in the FF case so long as $(\omega_0,\gamma_0)$ is not the endpoint of $\Gamma_0$. Since the convergence of $N(a,s)/(1-s^\frac{p-1}{2})$ is uniform in $s$ there is therefore an $\eps>0$ and $\delta>0$ such that, for $\omega<\omega_0$ and $\gamma<\gamma_0$ with $(\omega,\gamma)$ sufficiently close to $(\omega_0,\gamma_0)$, $N(a(\omega,\gamma),s)>\eps(1-s^\frac{p-1}{2})$ for all $s\in (1-\delta,1]$. Since $U(as)/as=O((1-s)^2)$ near $s=1$ when $\omega = \omega_0$, $\gamma=\gamma_0$, we then have
\begin{align*}
\lim_{\omega\to\omega_0^-, \gamma\to\gamma_0^-}\int_{1-\delta}^1 \frac{N(a,s)}{\left(\frac{U(as)}{as}\right)^\frac{3}{2}}ds \geq \int_{1-\delta}^1 \frac{\eps(1-s^\frac{p-1}{2})}{\left(\frac{U(as)}{as}\right)^\frac{3}{2}} ds= \infty.
\end{align*}
On the other hand, $D(a,s)$ is continuous in $a,\gamma$, and $s$. Since {\crm{$\frac{U(as)}{as}$}} is positive on $[0,1-\delta]$ for all $\omega,\gamma$, it follows that the infimum of $D(a,s)$ over $s\in [0,1-\delta]$, $\omega\in [\omega_0-\eps,\omega_0]$, $\gamma\in[\gamma_0-\eps,\gamma_0]$ is positive for some $\eps>0$. As $N(a(\omega,\gamma),s)$ is also bounded for $\omega,\gamma$ close to $\omega_0^-,\gamma_0^-$, the integral from $0$ to $1-\delta$ has a finite limit as $\omega\to\omega_0^-$, $\gamma\to \gamma_0^-$. Hence $\lim_{\omega\to \omega_0^-,\gamma\to \gamma_0^-} J(\omega,\gamma) = \infty$.

In the FF case, by Lemma \ref{lem:altcts}, we have $\lim_{\omega\to \omega^+,\gamma\to\gamma^+} a(\omega,\gamma)= b_0$ for some $b_0>a_0$ with $U'(b_0)<0$. Using the Iliev-Kirchev formula (\ref{I-KJ}), we note that
\begin{align*}
3\frac{U(s)}{s}+U'(a(\omega,\gamma))-U'(s)\to 3\frac{U(s)}{s}+U'(b_0)-U'(s)
\end{align*}
as $\omega\nearrow\omega_0$, $\gamma\nearrow \gamma_0$, and that this convergence is uniform on $[0,1]$. At $s=a_0$, the right hand side is $U'(b_0)<0$, so there are $\delta,\eps>0 $ such that 
\begin{align*}
\int_{a_0-\delta}^{a_0+\delta}\frac{3U(s)/s+U'(a)-U'(s)}{\left(\frac{U(as)}{as}\right)^\frac{3}{2}}ds<\int_{a_0-\delta}^{a_0+\delta}\frac{-\eps}{\left(\frac{U(as)}{as}\right)^\frac{3}{2}}ds.
\end{align*}
for $\omega<\omega_0$, $\gamma<\gamma_0$ close to $(\omega_0,\gamma_0)$. Since $U(\omega_0,\gamma_0,s)$ has a double zero at $s=a_0$, the limit of the right hand side is $-\infty$ for $\omega\nearrow \omega_0$, $\gamma\nearrow\gamma_0$. Since $U'(\omega_0,\gamma_0,a_0)=U(\omega_0,\gamma_0,a_0)=0$, if $U(\omega_0,\gamma_0,c_0)=0$ for some $c_0\in (0,b_0)\setminus\{a_0\}$, then $U'(\omega_0,\gamma_0,s)$ would have at least four zeros in $[0,b_0]$. This contradicts Lemma \ref{lem:RoS}, so $U(\omega_0,\gamma_0,s)>0$ for $s\in(0,b_0)\setminus\{a_0\}$. Since $U'(b_0)<0$, the integrand of (\ref{I-KJ}) is $\Theta((b_0-s)^{1/2})$ near $b_0$,%
\footnote{The Landau notation $f(s)=\Theta(g(s))$ for $s$ near $b_0$ means 
$C^{-1} g(s) \le f(s) \le C g(s)$ for $s$ near $b_0$ for some $C>0$.}
and is therefore uniformly integrable on $[0,a_0-\delta]\cup[a_0+\delta,b_0]$. Hence $\lim_{\omega\to\omega^+,\gamma\to\gamma^+}J(\omega,\gamma) = -\infty$.
\end{proof}

\section{Theorems for the FF Case}\label{S6}
{\crm In this section we prove results for the FF case with $a_1=a_3=1$.}
\begin{proposition}\label{prop:FFlimits}
The limits of $J(\omega,\gamma)$ for $\omega\to 0,\infty$ and $\gamma\to -\infty,\infty$ are as follows:
\begin{enumerate}
\item\label{case:FFomsml}
\begin{enumerate}
\item
If $p>5$, then $\lim_{\omega\to 0} J(\omega,\gamma)=-\infty$ for all $\gamma\in \R$. 
\item
If $p=5$, $\gamma\neq 0$, there are four cases:
\begin{enumerate}
\item If $q>9$, then $\lim_{\omega\to0} J(\omega, \gamma) =0^{\sign(\gamma)} $.
\item If $q = 9$ and $\gamma>0$, then $\lim_{\omega\to 0} J(\omega,\gamma) \in (0,\infty)$.
\item If $q = 9$ and $\gamma<0$, then $\lim_{\omega\to 0} J(\omega,\gamma) \in (-\infty,0)$.
\item If $q<9$, then $\lim_{\omega\to0} J(\omega, \gamma) =\sign(\gamma)\infty$.
\end{enumerate}
\item
If $p=5$, $\gamma=0$, there are three cases:
\begin{enumerate}
\item If $r>9$, then $\lim_{\omega\to0} J(\omega, \gamma) =0^- $.
\item If $r = 9$, then $\lim_{\omega\to 0} J(\omega,\gamma) \in (-\infty,0)$.
\item If $r<9$, then $\lim_{\omega\to0} J(\omega, \gamma) =-\infty$.
\end{enumerate}
\item
If $\frac{7}{3}<p<5$, then $\lim_{\omega\to0} J(\omega,\gamma) =\infty$ for all $\gamma\in \R$. 
\item
If $p=\frac{7}{3}$, then $\lim_{\omega\to0} J(\omega,\gamma)\in (0,\infty)$ for all $\gamma\in \R$. 
\item
If $p<\frac{7}{3}$, then $\lim_{\omega\to0} J(\omega,\gamma)=0^+$ for all $\gamma\in \R$. 
\end{enumerate}
\item\label{case:FFomlrg}
\begin{enumerate}
\item 
If $r>5$, then $\lim_{\omega\to \infty} J(\omega,\gamma)=0^-$ for all $\gamma\in \R$. 
\item
 If $r=5$, then $\lim_{\omega\to\infty} J(\omega, \gamma) =0^-$ for $\gamma>0$, and $\lim_{\omega\to 0} J(\omega,\gamma)=0^+ $ for $\gamma\leq 0$.
\item
If $\frac{7}{3}<r<5$, then $\lim_{\omega\to\infty} J(\omega,\gamma) =0^+$ for all $\gamma\in \R$.
\item
If $r=\frac{7}{3}$, then $\lim_{\omega\to\infty} J(\omega,\gamma)\in (0,\infty)$ for all $\gamma\in \R$.
\item
If $r<\frac{7}{3}$, then $\lim_{\omega\to\infty} J(\omega,\gamma)=\infty$ for all $\gamma\in \R$.
\end{enumerate}
\item\label{case:FFgamlrg}
\begin{enumerate}
\item If $r<\frac{7}{3}$, then $\lim_{\gamma\to\infty}J(\omega,\gamma)=\infty$ for all $ \omega>0$. 
\item If $r=\frac{7}{3}$, then $\lim_{\gamma\to\infty}J(\omega,\gamma)\in (0,\infty)$ for all $ \omega>0$. 
\item If $r>\frac{7}{3}$ and $r+2q<7$, then $\lim_{\gamma\to\infty}J(\omega,\gamma)=0^+$ for all $\omega>0$. 
\item If $q<7/3$ and $r+2q=7$ then $\lim_{\gamma\to\infty}J(\omega,\gamma)=0$ for all $\omega>0$. 
\item If $q<7/3$ and $r+2q>7$, then $\lim_{\gamma\to\infty} J(\omega,\gamma)=0^-$ for all $\omega>0$.
\item If $q\ge 7/3$, then $\lim_{\gamma\to\infty} J(\omega,\gamma)=-\infty$ for all $\omega>0$.
\end{enumerate}
\item\label{case:FFgamsml}
\begin{enumerate}
\item If $q\leq5$, then $\lim_{\gamma\to-\infty}J(\omega,\gamma)=0^+$ for all $\omega>0$. 
\item If $q>5$, then $\lim_{\gamma\to-\infty} J(\omega,\gamma)= 0^-$ for all $\omega>0$.
\end{enumerate}
\end{enumerate}
\end{proposition}
\begin{proof}
First consider the case $\omega \to 0^+$. We have $a \to 0^+$ by Lemma \ref{lem:aomega}.
Suppose $p\neq 5$. Factoring out $a^\frac{p-1}{2}$ from Lemma \ref{lem:pqrstab} gives
\begin{align*}
J(\omega,\gamma) &= \frac{C(\omega,\gamma)}{a^\frac{p-1}{4}}\int_0^1\frac{\frac{a_1(5-p)}{p+1}(1-s^\frac{p-1}{2})-\frac{\gamma(5-q)}{q+1}(1-s^\frac{q-1}{2})a^\frac{q-p}{2}+\frac{a_3(5-r)}{r+1}(1-s^\frac{r-1}{2})a^\frac{r-p}{2}}{\left(\frac{a_1}{p+1}(1-s^\frac{p-1}{2})-\frac{\gamma}{q+1}(1-s^\frac{q-1}{2})a^\frac{q-p}{2}+\frac{a_3}{r+1}(a-s^\frac{r-1}{2})a^\frac{r-p}{2}\right)^\frac{3}{2}}\\
&=\frac{(5-p)C(\omega,\gamma)}{a^\frac{p-1}{4}}\left(\int_0^1\left(\frac{p+1}{a_1(1-s^\frac{p-1}{2})}\right)^\frac{1}{2}+o(1)\right),
\end{align*}
where $o(1) \to 0$ as $a \to 0^+$.
When $p=5$, the first term in the numerator vanishes. For $p=5$, $\gamma\neq 0$, 
\begin{align*}
J(\omega,\gamma) = -\gamma(5-q)a^\frac{2q-3p+1}{4}C(\omega,\gamma)\left(\int_0^1\frac{\frac{1}{q+1}(1-s^\frac{q-1}{2})}{\left(\frac{a_1}{p+1}(1-s^\frac{p-1}{2})\right)^{3/2}}+o(1)\right).
\end{align*}
And when $p=5$, $\gamma = 0$,
\begin{align*}
J(\omega,\gamma) = a_3(5-r)a^\frac{2r-3p+1}{4}C(\omega,\gamma)\left(\int_0^1\frac{\frac{1}{r+1}(1-s^\frac{r-1}{2})}{\left(\frac{a_1}{p+1}(1-s^\frac{p-1}{2})\right)^{3/2}}ds+o(1)\right).
\end{align*}
For the asymptotic behaviour of $U'(a)$, we use $F_1(a)=\omega$ and get
\begin{align*}
U'(a) = \frac{2a_1}{p+1}a^\frac{p-1}{2}-\frac{2\gamma}{q+1}a^\frac{q-1}{2}+\frac{2a_3}{r+1}a^\frac{r-1}{2}-a_1a^\frac{p-1}{2}+\gamma a^\frac{q-1}{2} - a_3a^\frac{r-1}{2} = -\Theta(a^\frac{p-1}{2}).
\end{align*}
Thus $C(\omega,\gamma) = \Theta(a^\frac{3-p}{2})$ as $a\to 0$.
Altogether we have, for $p\neq 5$,
\begin{align*}
J(\omega,\gamma) = (5-p)\Theta(a^\frac{7-3p}{4}),
\end{align*}
for $p=5$, $\gamma\neq 0$,
\begin{align*}
J(\omega,\gamma) = -\gamma(5-q)\Theta(a^\frac{q-9}{2}),
\end{align*}
and for $p=5$, $\gamma = 0$,
\begin{align*}
J(\omega,\gamma) = a_3(5-r)\Theta(a^\frac{r-9}{2}).
\end{align*}
This proves part \ref{case:FFomsml}.

For the large $\omega$ case, 
we have $a \to \infty$ as $\omega \to \infty$ by Lemma \ref{lem:aomega}.
We factor out $a^\frac{r-1}{2}$ from Lemma \ref{lem:pqrstab} to get, for $r\neq 5$,
\begin{align*}
J(\omega,\gamma) &= \frac{C(\omega,\gamma)}{a^\frac{r-1}{4}}\int_0^1\frac{\frac{a_1(5-p)}{p+1}(1-s^\frac{p-1}{2})a^\frac{p-r}{2}-\frac{\gamma(5-q)}{q+1}(1-s^\frac{q-1}{2})a^\frac{q-r}{2}+\frac{a_3(5-r)}{r+1}(1-s^\frac{r-1}{2})}{\left(\frac{a_1}{p+1}(1-s^\frac{p-1}{2})a^\frac{p-r}{2}-\frac{\gamma}{q+1}(1-s^\frac{q-1}{2})a^\frac{q-r}{2}+\frac{a_3}{r+1}(a-s^\frac{r-1}{2})\right)^\frac{3}{2}}\\
&=\frac{(5-r)C(\omega,\gamma)}{a^\frac{r-1}{4}}\left(\int_0^1\left(\frac{r+1}{a_3(1-s^\frac{r-1}{2})}\right)^\frac{1}{2}ds+o(1)\right),
\end{align*}
where $o(1) \to 0$ as $a \to \infty$.
When $r=5$ and $\gamma\neq 0$,
\begin{align*}
J(\omega,\gamma) = -\gamma(5-q)a^\frac{2q-3r+1}{4}C(\omega,\gamma)\left(\int_0^1\frac{\frac{1}{q+1}(1-s^\frac{q-1}{2})}{\left(\frac{a_3}{r+1}(1-s^\frac{r-1}{2})\right)^{3/2}}ds+o(1)\right).
\end{align*}
And when $r=5$, $\gamma = 0$,
\begin{align*}
J(\omega,\gamma) = a_1(5-p)a^\frac{2p-3r+1}{4}C(\omega,\gamma)\left(\int_0^1\frac{\frac{1}{p+1}(1-s^\frac{p-1}{2})}{\left(\frac{a_3}{r+1}(1-s^\frac{r-1}{2})\right)^{3/2}}ds+o(1)\right).
\end{align*}
For the asymptotics of $U'(a)$ as $a\to\infty$, we use $F_1(a) = \omega$, and get
\begin{align*}
U'(a) = \frac{2a_1}{p+1}a^\frac{p-1}{2}-\frac{2\gamma}{q+1}a^\frac{q-1}{2}+\frac{2a_3}{r+1}a^\frac{r-1}{2}-a_1a^\frac{p-1}{2}+\gamma a^\frac{q-1}{2} - a_3a^\frac{r-1}{2} = -\Theta(a^\frac{r-1}{2}).
\end{align*}
Thus $C = \Theta(a^\frac{3-r}{2})$. Altogether, for $r\neq 5$,
\begin{align*}
J(\omega,\gamma) = (5-r)\Theta(a^\frac{7-3r}{4}),
\end{align*}
for $r=5$, $\gamma\neq 0$,
\begin{align*}
J(\omega,\gamma) = -\gamma(5-q)\Theta(a^\frac{q-9}{2}),
\end{align*}
and for $r=5$, $\gamma = 0$,
\begin{align*}
J(\omega,\gamma) = a_1(5-p)\Theta(a^\frac{p-9}{2}).
\end{align*}
This proves part \ref{case:FFomlrg}.

For the large $\gamma$ case, fix $\omega>0$. By Lemma \ref{lem:agamma}, we may equivalently consider the limit as $a\to\infty$ with $\gamma$ as a function of $a$. As 
\begin{align*}
    \omega = F_1(a) = \frac{2a_1 a^\frac{p-1}{2}}{p+1}-\frac{2\gamma a^\frac{q-1}{2}}{q+1}+\frac{2a_3a^\frac{r-1}{2}}{r+1}
\end{align*}
we have $\lim_{a\to \infty}\frac{\gamma}{q+1}a^\frac{q-r}{2}=\frac{1}{r+1}$. If $q<\frac{7}{3}$, then, as $a\to\infty$,{\crm
\begin{align*}
J(\omega,\gamma) &= \frac{C(\omega,\gamma)}{a^\frac{r-1}{4}}\int_0^1\frac{\frac{a_1(5-p)}{p+1}(1-s^\frac{p-1}{2})a^\frac{p-r}{2}-\frac{\gamma(5-q)}{q+1}(1-s^\frac{q-1}{2})a^\frac{q-r}{2}+\frac{(5-r)}{r+1}(1-s^\frac{r-1}{2})}{\left(\frac{a_1}{p+1}(1-s^\frac{p-1}{2})a^\frac{p-r}{2}-\frac{\gamma}{q+1}(1-s^\frac{q-1}{2})a^\frac{q-r}{2}+\frac{1}{r+1}(1-s^\frac{r-1}{2})\right)^\frac{3}{2}}\\
&=\frac{C(\omega,\gamma)(r+1)^\frac{1}{2}}{a^\frac{r-1}{4}}\left(\int_0^1\frac{-(5-q)(1-s^\frac{q-1}{2})+(5-r)(1-s^\frac{r-1}{2})}{\left(s^\frac{q-1}{2}-s^\frac{r-1}{2}\right)^{3/2}}+o(1)\right)\\
&=\frac{C(\omega,\gamma)(r+1)^\frac{1}{2}}{a^\frac{r-1}{4}}\left(2\frac{7-2q-r}{r-q}B\left(\frac{7-3q}{2(r-q)},\frac{1}{2}\right)+o(1)\right)
\end{align*}}
where $o(1) \to 0$ as $a \to \infty$, and
the last equality is from Lemma \ref{lem:2pow}. Since $\gamma a^\frac{q-r}{2}\to \frac{q+1}{r+1}$, $U'(a)$ is $-\Theta(a^\frac{r-1}{2})$,
 and $C(\omega,\gamma) = \Theta(a^\frac{3-r}{2})$. Thus, for $q<\frac{7}{3}$,
\begin{align*}
J(\omega,\gamma) = (7-2q-r)\Theta(a^\frac{7-3r}{4}).
\end{align*}
If $q\geq \frac{7}{3}$, then (\ref{pqrstab}) is uniformly integrable at $1$, but not at $0$. Since the numerator of the integrand is negative for $s$ close to $0$, we have $J(\omega,\gamma)\to -\infty$ as $\gamma\to \infty$ in this case.
 This proves part \ref{case:FFgamlrg}.

For large $-\gamma$ and fixed $\omega$, by Lemma \ref{lem:agamma}, we may equivalently consider the limit as $a\to 0$ for fixed $\omega$. As $\omega = F_1(a)$, we have $-\frac{\gamma}{q+1}a^\frac{q-1}{2}\to \frac{\omega}{2}$ as $a\to 0$. Thus, for $q\neq 5$,
\begin{align*}
J(\omega,\gamma) &=  C(\omega,\gamma)\int_0^1\frac{\frac{a_1(5-p)}{p+1}(1-s^\frac{p-1}{2})a^\frac{p-1}{2}-\frac{\gamma(5-q)}{q+1}(1-s^\frac{q-1}{2})a^\frac{q-1}{2}+\frac{a_3(5-r)}{r+1}(1-s^\frac{r-1}{2})a^\frac{r-1}{2}}{\left(\frac{a_1}{p+1}(1-s^\frac{p-1}{2})a^\frac{p-1}{2}-\frac{\gamma}{q+1}(1-s^\frac{q-1}{2})a^\frac{q-1}{2}+\frac{a_3}{r+1}(1-s^\frac{r-1}{2})a^\frac{r-1}{2}\right)^\frac{3}{2}}ds\\
&=\frac{\sqrt{2}(5-q)C(\omega,\gamma)}{\sqrt{\omega}}\left(\int_0^1(1-s^\frac{q-1}{2})^{-\frac{1}{2}}ds+o(1)\right)
\end{align*}
where $o(1) \to 0$ as $a \to 0$.
For $q=5$, factoring out $a^\frac{p-1}{2}$ from the numerator gives
\begin{align*}
J(\omega,\gamma) = \frac{a_1(5-p)a^\frac{p-1}{2}C(\omega,\gamma)}{(p+1)(\omega/2)^\frac{3}{2}}\left(\int_0^1\frac{(1-s^\frac{p-1}{2})}{(1-s^\frac{q-1}{2})^\frac{3}{2}}ds + o(1)\right)
\end{align*}
As $\gamma a^\frac{q-1}{2}\to \frac{-(q+1)\omega}{2}$, we have 
\begin{align*}
    U'(a) = \omega - a_1a^\frac{p-1}{2}+\gamma a^\frac{q-1}{2}-a_3a^\frac{r-1}{2} = -\Theta(1)
\end{align*}
as $a\to 0$. Thus $C(\omega,\gamma)= \Theta(a)$, and so, for $q\neq 5$,
\begin{align*}
J(\omega,\gamma) = (5-q)\Theta(a)
\end{align*}
and for $q=5$
\begin{align*}
J(\omega,\gamma) = a_1\Theta(a^\frac{p+1}{2}).
\end{align*}
This proves part \ref{case:FFgamsml}.
\end{proof}

Part \ref{case:FFomlrg} of the proposition above shows that, for $r>5$, there is a function $\omega_-(\gamma)$ such that $J(\omega,\gamma)<0$ for all $\omega>\omega_-(\gamma)$. The following proposition shows when this bound can be made uniform in $\gamma$. Note that Proposition \ref{prop:FFlimits} part \ref{case:FFgamsml} shows that there is no uniform bound when $q\leq 5$.
\begin{proposition}\label{prop:FFbddstab}
If $q>5$, then there is an $\omega_->0$ such that $J(\omega,\gamma)<0$ for all $\omega>\omega_-$, $\gamma\in\R$.
\end{proposition}
\begin{proof}
Fix any $\omega_1>0$.
We first show that 
there is a $\gamma_-\in\R$ such that $J(\omega,\gamma)<0$ for all $\omega>\omega_1$ and $\gamma<\gamma_-$. If $p\geq5$, then {\crm$N(a,s)$} is negative for all $\gamma<0$ and $\omega>0$, so the claim follows. Suppose $p<5<q<r$. As in the proof of Proposition \ref{prop:FFlimits}, we have $\gamma a(\omega_1,\gamma)^\frac{q-1}{2}\to \frac{-(q+1)\omega_1}{2}$ and $a(\omega_1,\gamma)\to 0$ as $\gamma\to -\infty$. Hence, there is a $\gamma_-<0$ such that the second term in
\begin{align*}
    {\crm N(a(\omega_1,\gamma),s)} = a_1\frac{5-p}{p+1}(1-s^\frac{p-1}{2})a(\omega_1,\gamma)^\frac{p-1}{2}&-\gamma\frac{5-q}{q+1}(1-s^\frac{q-1}{2})a(\omega_1,\gamma)^\frac{q-1}{2}
    \\
    & +a_3\frac{5-r}{r+1}(1-s^\frac{r-1}{2})a(\omega_1,\gamma)^\frac{r-1}{2}
\end{align*}
dominates, $N(a(\omega_1,\gamma),s)<0$, for all $s\in(0,1)$ and $\gamma<\gamma_-$. Fixing $\gamma<\gamma_-$ and $s\in(0,1)$, we have $N(a,s)=\Theta(a^\frac{p-1}{2})>0$ as $a\to 0$ and $N(a(\omega_1,\gamma),s)<0$. Since the last two terms in $N(a,s)$ have negative coefficients, by Lemma \ref{lem:RoS}, $N(a,s)$ cannot change signs twice for $a\in (0,\infty)$. Hence $N(a,s)<0$ for all $a>a(\omega_1,\gamma)$, and hence (using $a$ is increasing in $\omega$),
 $N(a(\omega,\gamma),s)<0$ for all $\omega>\omega_1$ and $s\in (0,1)$. Therefore $J(\omega,\gamma)<0$ for all $\omega>\omega_1$, which proves the claim.

Similarly, we show that, for the same fixed $\omega_1>0$, there is a $\gamma_+>0$ such that $J(\omega,\gamma)<0$ for all $\omega>\omega_1 $ and $\gamma>\gamma_+$. As in the proof of Proposition \ref{prop:FFlimits}, we have $\frac{\gamma}{q+1}a^\frac{q-r}{2}\to\frac{1}{r+1}$ as $a\to \infty$ for fixed $\omega_1$. 
Since $a(\omega_1,\gamma)$ is increasing in $\gamma$, it follows that there is a $\gamma_0>0$ such that $a^\frac{r-q}{2}>\gamma$ for all $\gamma>\gamma_0$. Using the fact that $1-s^\frac{q-1}{2}<1-s^\frac{r-1}{2}$, we now have, for $\gamma>\gamma_0$ {\crm and $a=a(\omega_1,\gamma)$,}
\begin{align*}
-\gamma(5-q){\crm A_q(a,s)} + (5-r){\crm A_r(a,s)}&=-\frac{\gamma(5-q)}{q+1}(1-s^\frac{q-1}{2})a^\frac{q-1}{2}+\frac{5-r}{r+1}(1-s^\frac{r-1}{2})a^\frac{r-1}{2}\\
&<\left(\frac{5-r}{r+1}-\frac{5-q}{q+1}\right)(1-s^\frac{r-1}{2})a^\frac{r-1}{2}.
\end{align*}
For fixed $s\in (0,1)$, $\gamma>\gamma_0$, and considering {\crm$-\gamma(5-q)A_q(a,s) + (5-r)A_r(a,s)$} as a function of $a$, we also have {\crm $-\gamma(5-q)A_q(a,s) + (5-r)A_r(a,s)=\Theta(a^\frac{q-1}{2})$} as $a\to 0$. By Lemma \ref{lem:RoS}, $-\gamma(5-q)A_q(a,s) + (5-r)A_r(a,s)$ changes sign only once for $a\in (0,\infty)$, so we have $-\gamma(5-q)A_q(a,s) + (5-r)A_r(a,s)<0$ for all $a>a(\omega_1,\gamma)$. Since this holds for all $s\in (0,1)$, we now have $-\gamma(5-q)A_q(a,s) + (5-r)A_r(a,s)<0$ for all $s\in (0,1)$, $\omega>\omega_1$, and $\gamma>\gamma_0$. If $p\geq 5$, then $(5-p)A_p(a,s)\leq 0$, so this suffices to show that $J(\omega,\gamma)<0$ for all $\omega>\omega_1$ and $\gamma>\gamma_+=\gamma_0$.

Now suppose $p<5$. Since $a\to\infty$ as $\gamma\to\infty$, we can find $\gamma_+>\gamma_0$ such that
\begin{align*}
    \left(\frac{5-r}{r+1}-\frac{5-q}{q+1}\right)(1-s^\frac{r-1}{2})a(\omega_1,\gamma)^{\frac{r-1}2}<-(5-p)A_p(a(\omega_1,\gamma)s)
\end{align*}
for all $s\in (0,1)$ and $\gamma  > \gamma_+$. For fixed $s\in (0,1)$ and $\gamma>\gamma_+$, $N(a,s)=\Theta(a^\frac{p-1}{2})>0$ for small $a>0$, and, by Lemma \ref{lem:RoS}, changes sign only once for $a\in (0,\infty)$. The numerator is therefore negative for all $s\in (0,1)$ and $a>a(\omega_1,\gamma)$. Hence $J(\omega,\gamma)<0$ for all $\omega>\omega_1$ and $\gamma>\gamma_+$.

Now consider $J(\omega,\gamma)$ for $\gamma\in [\gamma_-,\gamma_+]$. Since $a\to\infty$ as $\omega \to\infty$, there is a $\omega_2$ such that
\begin{align*}
    {\crm (5-r)A_r(a(\omega_2,\gamma),s)<-(5-p)A_p(a(\omega_2,\gamma)s)+\gamma(5-q)A_q(a(\omega_2,\gamma),s)}
\end{align*}
for all $s\in (0,1)$, $\gamma\in [\gamma_-,\gamma_+]$, and $\omega>\omega_2$. Hence $J(\omega,\gamma)<0$ for all $\omega>\omega_2$, $\gamma\in [\gamma_-,\gamma_+]$, and hence $J(\omega,\gamma)<0$ for all $\omega>\omega_- = \max\{\omega_1,\omega_2\}$, $\gamma\in\R$.
\end{proof}

\section{Theorems for the FD Case}\label{S7}

{\crm In this section we prove results for the FD case with $a_1=1$ and $a_3=-1$.}

By Proposition \ref{prop:Gamlimits}, we have {\crm$\lim_{\omega\to \omega_0^-} J(\omega,\gamma_0)=\infty$} and {\crm$\lim_{\gamma\to \gamma_0^-} J(\omega_0,\gamma)=\infty$} for any {\crm$(\omega_0,\gamma_0)\in\Gamma_{ne}$ (which has no endpoint in the FD case)}. The limits for small $\omega$ and large $-\gamma$ are computed in the same way as the FF case.
\begin{proposition}\label{prop:FDlimits}
The limits of $J(\omega,\gamma)$ for $\omega\to 0$ and $\gamma\to -\infty$ are as follows:
\begin{enumerate}
\item 
\begin{enumerate}
\item
If $p>5$, then $\lim_{\omega\to 0} J(\omega,\gamma)=-\infty$ for all $\gamma\in \R$. 
{\crm
\item
If $p=5$, $\gamma\neq 0$, there are four cases:
\begin{enumerate}
    \item If $q>9$, then $\lim_{\omega\to 0}J(\omega,\gamma)=0^{\sign(\gamma)}$.
    \item If $q=9$ and $\gamma>0$, then $\lim_{\omega\to 0}J(\omega,\gamma)\in (0,\infty)$.
    \item If $q=9$ and $\gamma<0$, then $\lim_{\omega\to 0}J(\omega,\gamma)\in (-\infty,0)$.
    \item If $q<9$, then $\lim_{\omega\to 0}J(\omega,\gamma)=\sign(\gamma)\infty$.
\end{enumerate}
\item 
If $p=5$, $\gamma=0$, there are three cases:
\begin{enumerate}
    \item If $r>9$, then $\lim_{\omega\to 0}J(\omega,\gamma)=0^+$.
    \item If $r=9$, then $\lim_{\omega\to 0}J(\omega,\gamma)\in(0,\infty)$.
    \item If $r<9$, then $\lim_{\omega\to 0}J(\omega,\gamma)=\infty$.
\end{enumerate}
}
\item
If $\frac{7}{3}<p<5$, then $\lim_{\omega\to0} J(\omega,\gamma) =\infty$ for all $\gamma\in \R$. 
\item
If $p=\frac{7}{3}$, then $\lim_{\omega\to0} J(\omega,\gamma)\in (0,\infty)$ for all $\gamma\in \R$. 
\item
If $p<\frac{7}{3}$, then $\lim_{\omega\to0} J(\omega,\gamma)=0^+$ for all $\gamma\in \R$. 
\end{enumerate}
\item
\begin{enumerate}
\item If $q\leq5$, then $\lim_{\gamma\to-\infty}J(\omega,\gamma)=0^+$ for all $\omega>0$. 
\item If $q>5$, then $\lim_{\gamma\to-\infty} J(\omega,\gamma)= 0^-$ for all $\omega>0$.
\end{enumerate}
\end{enumerate}
\end{proposition}
Since the limits of $J$ close to the nonexistence curve are positive, the stable region is nonempty for all $1<p<q<r$. By Proposition \ref{prop:FDlimits} above, the unstable region is nonempty when $q>5$. Conversely, we can show that unstable region is empty for $q\leq 5$.
\begin{proposition}\label{prop:FDallstab}
If $q\leq 5$, then $J(\omega,\gamma)>0$ for all $(\omega,\gamma)\in R_{\ex}$.
\end{proposition}
\begin{proof}
For any $\gamma\in \R$, let $a^*(\gamma) = a(\omega^*(\gamma),\gamma)$, so that $U'(a^*(\gamma))=0$. First suppose $r\leq5$ and fix $\gamma>0$. Using the fact that $\frac{1-s^\frac{l-1}{2}}{1-s^\frac{p-1}{2}}\leq\frac{l-1}{p-1}$ for all $s\in (0,1)$ and $l=q,r$, we have
\[
\frac{N(a,s)}{1-s^\frac{p-1}{2}}\geq\frac{1}{p-1}\left(\frac{(5-p)(p-1)}{p+1}a^\frac{p-1}{2}-\gamma\frac{(5-q)(q-1)}{q+1}a^\frac{q-1}{2}-\frac{(5-r)(r-1)}{r+1}a^\frac{r-1}{2}\right).\\
\]
For $a=a^*$, eliminating $\gamma$ using the parameterization in Proposition \ref{prop:Rex} gives
\begin{align*}
\frac{N(a^*,s)}{1-s^\frac{p-1}{2}}\geq \frac{1}{p-1}\left(\frac{p-1}{p+1}(q-p){a^*}^\frac{p-1}{2}+\frac{r-1}{r+1}(r-q){a^*}^\frac{r-1}{2}\right)>0.
\end{align*}
We also have $N(a,s)>0$ for small $a>0$. By Lemma \ref{lem:RoS}, $N(a,s)$ changes sign at most once for $a\in (0,a^*)$, so it follows that $N(a,s)\geq 0$ for all $a\in (0,a^*)$. Since this holds for each $s\in (0,1)$, $J(\omega,\gamma)>0$ for all $\omega<\omega^*(\gamma)$.

Next, suppose $\gamma>0$ and $r>5$. Since $r>5$, it suffices to show that
the first two terms of $N(a,s)$ in \eqref{ND.def}
\[
N_{1,2}(a,s) :=
\frac{5-p}{p+1}(1-s^\frac{p-1}{2})a^\frac{p-1}{2}-\gamma\frac{5-q}{q+1}(1-s^\frac{q-1}{2})a^\frac{q-1}{2}>0
\]
for all $s\in (0,1)$. Indeed, using $\frac{1-s^\frac{q-1}{2}}{1-s^\frac{p-1}{2}}\leq \frac{q-1}{p-1}$ and the parameterization from Proposition \ref{prop:Rex},
\begin{align*}
\frac{N_{1,2}(a^*,s)}{1-s^\frac{p-1}{2}} &>\frac{5-p}{p+1}{a^*}^\frac{p-1}{2}-\gamma\frac{(5-q)(q-1)}{(q+1)(p-1)}{a^*}^\frac{q-1}{2}\\
&=\frac{q-p}{p+1}{a^*}^\frac{p-1}{2}+\frac{(5-q)(r-1)}{(p-1)(r+1)}a^\frac{r-1}{2}>0
\end{align*}
Since $N_{1,2}(a,s)$ is also positive for small $a$ and $N_{1,2}(a,s)$ and changes sign at most once for $a\in (0,a^*)$, we have $N_{1,2}(a,s)>0$ for all $a\in (0,a^*)$. Hence $J(\omega,\gamma)>0$ for $\omega<\omega^*(\gamma)$.

Now suppose $\gamma\leq 0$. If $r\geq5$, then each term in $N(a,s)$ is positive, so we are done. Suppose $r<5$. Since $\frac{1-s^\frac{t-1}{2}}{1-s^\frac{r-1}{2}}\geq\frac{t-1}{r-1}$ for $t<r$,
\begin{align*}
\frac{N(a,s)}{1-s^\frac{r-1}{2}}\geq\frac{1}{r-1}\left(\frac{(5-p)(p-1)}{p+1}a^\frac{p-1}{2}-\gamma\frac{(5-q)(q-1)}{q+1}a^\frac{q-1}{2}-\frac{(5-r)(r-1)}{r+1}a^\frac{r-1}{2}\right).
\end{align*}
Using the parameterization in Proposition \ref{prop:Rex}, we have
\begin{align*}
\frac{N(a^*,s)}{1-s^\frac{r-1}{2}}\geq\frac{1}{r-1}\left(\frac{p-1}{p+1}(q-p){a^*}^\frac{p-1}{2}+\frac{r-1}{r+1}(r-q){a^*}^\frac{r-1}{2}\right)>0.
\end{align*}
As in the case for $\gamma>0$, $N(a,s)>0$ for small $a$ and $N(a,s)$ changes sign at most once for $a\in (0,a^*)$. It follows that $N(a,s)\geq 0$ for all $a\in (0,a^*)$, $s\in (0,1)$, and hence $J(\omega,\gamma)>0$ for all $\omega\in (0,\omega^*(\gamma))$.
\end{proof}

\section{Theorems for the DF Case}\label{S8}

{\crm In this section we prove results for the DF case with $a_1=-1$ and $a_3=1$.}

The limits of $J(\omega,\gamma)$ for $\omega\to \infty$ and $\gamma\to \pm \infty$ are calculated in the same way as the FF case.
\begin{proposition}\label{prop:DFlimits}
The limits of $J(\omega,\gamma)$ for $\omega\to\infty$ and $\gamma\to -\infty,\infty$ are as follows:
\begin{enumerate}
\item 
\begin{enumerate}
\item 
If $r>5$, then $\lim_{\omega\to \infty} J(\omega,\gamma)=0^-$ for all $\gamma\in \R$. 
\item
 If $r=5$, then $\lim_{\omega\to\infty} J(\omega, \gamma) =0^-$ for $\gamma\geq 0$, and $\lim_{\omega\to 0} J(\omega,\gamma)=0^+ $ for $\gamma< 0$.
\item
If $\frac{7}{3}<r<5$, then $\lim_{\omega\to\infty} J(\omega,\gamma) =0^+$ for all $\gamma\in \R$.
\item
If $r=\frac{7}{3}$, then $\lim_{\omega\to\infty} J(\omega,\gamma)\in (0,\infty)$ for all $\gamma\in \R$.
\item
If $r<\frac{7}{3}$, then $\lim_{\omega\to\infty} J(\omega,\gamma)=\infty$ for all $\gamma\in \R$.

\end{enumerate}
\item
\begin{enumerate}
\item If $r<\frac{7}{3}$, then $\lim_{\gamma\to\infty}J(\omega,\gamma)=\infty$ for all $ \omega>0$. 
\item If $r=\frac{7}{3}$, then $\lim_{\gamma\to\infty}J(\omega,\gamma)\in (0,\infty)$ for all $ \omega>0$. 
\item If $r>\frac{7}{3}$ and $r+2q<7$, then $\lim_{\gamma\to\infty}J(\omega,\gamma)=0^+$ for all $\omega>0$. 
\item If $r+2q=7$ then $\lim_{\gamma\to\infty}J(\omega,\gamma)=0$ for all $\omega>0$. 
\item If $r+2q>7$, then $\lim_{\gamma\to\infty} J(\omega,\gamma)=0^-$ for all $\omega>0$.
\end{enumerate}
\item
\begin{enumerate}
\item If $q<5$, then $\lim_{\gamma\to-\infty}J(\omega,\gamma)=0^+$ for all $\omega>0$. 
\item If $q\geq 5$, then $\lim_{\gamma\to-\infty} J(\omega,\gamma)= 0^-$ for all $\omega>0$.
\end{enumerate}
\end{enumerate}
\end{proposition}
Since $\lim_{\omega\to 0}a(\omega,\gamma)>0$ in the D* cases, the stability of solutions for small $\omega$ is not determined by whether $p<5$. When $p<\frac{7}{3}$, the integral in (\ref{pqrstab}) is uniformly integrable at both endpoints as $\omega\to 0$, so $\lim_{\omega\to 0} J(\omega,\gamma)=J(0,\gamma)$. For $\omega = 0$, we have $\frac{\gamma}{q+1}a^\frac{q-1}{2} = \frac{a_1}{p+1}a^\frac{p-1}{2} +\frac{a_3}{r+1}a^\frac{r-1}{2}$. Using this to eliminate the $\gamma$ dependency in (\ref{pqrstab}) gives,
\begin{align*}
J(0,\gamma)=C\left(\frac{p+1}{a^\frac{p-1}{2}}\right)^\frac{1}{2}\int_0^1\frac{N_1(s)+\beta N_2(s)}{\left(D_1(s)+\beta D_2(s)\right)^\frac{3}{2}}ds
\end{align*}
where $\beta = \frac{p+1}{r+1}a^\frac{r-p}{2}$ and
\begin{align*}
N_1(s)=a_1\left((5-p)(1-s^\frac{p-1}{2})-(5-q)(1-s^\frac{q-1}{2})\right),
\end{align*}
\begin{align*}
N_2(s)=a_3\left((5-r)(1-s^\frac{r-1}{2})-(5-q)(1-s^\frac{q-1}{2})\right),
\end{align*}
\begin{align*}
D_1(s)=a_1(s^\frac{q-1}{2}-s^\frac{p-1}{2}),\quad D_2(s)=a_3(s^\frac{q-1}{2}-s^\frac{r-1}{2}).
\end{align*}
{To determine the sign of the integral, we reformulate an inequality given in Alzer \cite{MR1388887} to get a bound on $\frac{\partial}{\partial x}B(x,1/2)$.}

\begin{lemma}\label{lem:Bxineq}
For all $b>0$, we have
\begin{align*}
-\frac{1}{2b}B(b+\frac{1}{2},\frac{1}{2})<\frac{\partial}{\partial x}B(b+\frac{1}{2},\frac{1}{2})<-\frac{1}{2b+1}B(b+\frac{1}{2},\frac{1}{2}).
\end{align*}
\end{lemma}
\begin{proof}
Recall that the beta and gamma functions are related by 
\begin{align*}
B(x,y)=\frac{\Gamma(x)\Gamma(y)}{\Gamma(x+y)}.
\end{align*}
Taking a derivative in $x$ gives
\begin{align*}
\frac{\partial}{\partial x}B(x,y) = B(x,y)\left(\frac{\Gamma'(x)}{\Gamma(x)}-\frac{\Gamma'(x+y)}{\Gamma(x+y)}\right)=B(x,y)\left(\psi(x)-\psi(x+y)\right)
\end{align*}
where $\psi(x) =\Gamma'(x)/\Gamma(x)$ is the digamma function. 
For any integer $n\geq 0$ and any $b>0$, $s\in (0,1)$, {\crm Theorem 7 in Alzer \cite{MR1388887} shows that}
\begin{align*}
    A_n(s,b)<\psi(b+1)-\psi(b+s)<A_n(s,b)+\delta_n(s,b),
\end{align*}
where $\lim_{n\to\infty}\delta_n(s,b)= 0$, and
\begin{align*}
    A_n(s,b)=(1-s)\left[\frac{1}{b+s+n}+\sum_{i=0}^{n-1}\frac{1}{(b+i+1)(b+i+s)}\right].
\end{align*}
In the case $s=1/2$,
\begin{align*}
    A_n(1/2,b)&=\frac{1}{2}\left[\frac{1}{b+n+1/2}+\frac{1}{(b+1)(b+1/2)}+2\sum_{i=1}^{n-1}\left(\frac{1}{b+i+1/2}-\frac{1}{b+i+1}\right)\right]\\
    &<\frac{1}{2}\left[\frac{1}{b+n+1/2}+\frac{1}{(b+1)(b+1/2)}+\sum_{i=1}^{n-1}\left(\frac{1}{b+i}-\frac{1}{b+i+1}\right)\right]\\
    &=\frac{1}{2}\left[\frac{1}{b+n+1/2}+\frac{1}{(b+1)(b+1/2)}+\frac{1}{b+1}-\frac{1}{b+n}\right]\\
    &<\frac{1}{2}\left[\frac{b+3/2}{(b+1)(b+1/2)}\right]<\frac{1}{2b}.
\end{align*}
For the lower bound, we have
\begin{align*}
    A_n(1/2,b)&=\frac{1}{2}\left[\frac{1}{b+n+1/2}+\frac{1}{(b+1)(b+1/2)}+2\sum_{i=1}^{n-1}\left(\frac{1}{b+i+1/2}-\frac{1}{b+i+1}\right)\right]\\
    &>\frac{1}{2}\left[\frac{1}{b+n+1/2}+\frac{1}{(b+1)(b+1/2)}+\sum_{i=1}^{n-1}\left(\frac{1}{b+i+1/2}-\frac{1}{b+i+3/2}\right)\right]\\
    &=\frac{1}{2}\left[\frac{1}{b+n+1/2}+\frac{1}{(b+1)(b+1/2)}+\frac{1}{b+3/2}-\frac{1}{b+n+1/2}\right]\\
    &= \frac{1}{2b+1}+\frac{1}{2}\left(\frac{1}{b+1/2}-\frac{2}{b+1}+\frac{1}{b+3/2}\right)>\frac{1}{2b+1}.
\end{align*}
As $\lim_{n\to\infty}\delta_n(1/2,b)=0$, this shows that
\begin{align*}
\frac{1}{2b+1}<\psi(b+1)-\psi(b+\frac{1}{2})<\frac{1}{2b}.
\end{align*}
Since $\frac{\partial}{\partial x}B(x,y)=(\psi(x)-\psi(x+y))B(x,y)$, the claim follows.
\end{proof}

\begin{proposition}\label{prop:DFJ0pos}
If $2q+r<7$, then $J(0,\gamma)>0$ for all $\gamma\in \R$.
\end{proposition}
\begin{proof}
Since $p<q<\frac{7}{3}$, $J(0,\gamma)$ is given by
\begin{align}
J(0,\gamma)=C\left(\frac{p+1}{a^\frac{p-1}{2}}\right)^\frac{1}{2}\int_0^1\frac{N_1(s)+\beta N_2(s)}{(D_1(s)+\beta D_2(s))^\frac{3}{2}}ds,\label{J_0}
\end{align}
which is integrable. To show that $J(0,\gamma)>0$, we consider each term $N_1, N_2$ separately. We have $N_1(0)<0$ and $N_1(1) =0$. By Lemma \ref{lem:RoS}, $N_1$ cannot have three positive zeros, so there is a $c\in (0,1]$ such that $N_1(s)\leq 0$ for $s\in [0,c)$ and $N_1(s)>0$ for $s\in (c,1)$. Since both terms $D_1$, $\beta D_2$ in the denominator are positive on $(0,1)$, we then have $\frac{N_1(s)}{(D_1(s)+\beta D_2(s))^\frac{3}{2}}\leq 0$ for $s\in(0,c)$ and $\frac{N_1(s)}{(D_1(s)+\beta D_2(s))^\frac{3}{2}}>0$ for $s\in(c,1)$. Now, if $\phi$ is a function that is positive and decreasing on $(0,1)$, then
\begin{align*}
\frac{N_1(s)}{\left(D_1(s)+\beta D_2(s)\right)^\frac{3}{2}}\phi(c) \geq\frac{N_1(s)}{\left(D_1(s)+\beta D_2(s)\right)^\frac{3}{2}}\phi(s)
\end{align*}
for all $s\in (0,1)$. Let $\phi(s) = \frac{(D_1(s)+\beta D_2(s))^\frac{3}{2}}{(s^\frac{q-1}{2}-s^{3-q})^\frac{3}{2}}$ with $\phi(1)=\lim_{s\to 1^-} \phi(s)$. Since $3-q>\frac{r-1}{2}$ we see that, using Lemma \ref{lem:lineq},
\begin{align*}
\phi'(s) = \frac{3}{2}\left(\phi(s)\right)^\frac{1}{3}\left(\frac{d}{ds}\frac{s^\frac{p-1}{2}-s^\frac{q-1}{2}}{s^\frac{q-1}{2}-s^{3-q}}+\beta\frac{d}{ds}\frac{s^\frac{q-1}{2}-s^\frac{r-1}{2}}{s^\frac{q-1}{2}-s^{3-q}}\right)<0.
\end{align*}
Hence, there is a $c\in (0,1]$ such that
\begin{align*}
\phi(c)\int_0^1\frac{N_1}{(D_1+\beta D_2)^\frac{3}{2}}ds&>\int_0^1\frac{N_1}{(s^\frac{q-1}{2}-s^{3-q})^\frac{3}{2}}ds\\
&=\int_0^1\frac{(5-q)(1-s^\frac{q-1}{2})-(5-p)(1-s^\frac{p-1}{2})}{(s^\frac{q-1}{2}-s^{3-q})^\frac{3}{2}}ds.
\end{align*}
Using a change of variables $t=s^\frac{7-3q}{2}$, we write this integral as
\begin{align*}
&\frac{2}{7-3q}\int_0^1\frac{(5-q)(1-t^\frac{q-1}{7-3q})-(5-p)(1-t^\frac{p-1}{7-3q})}{t^\frac{1}{2}(1-t)^\frac{3}{2}}dt\\
&=\frac{2(5-q)}{7-3q}H\left(\frac{1}{2},\frac{q-1}{7-3q}\right)-\frac{2(5-p)}{7-3q}H\left(\frac{1}{2},\frac{p-1}{7-3q}\right),\\
\end{align*}
and using Lemma \ref{lem:BvH}, this becomes
\begin{align}
\frac{4(5-q)(q-1)}{(7-3q)^2}B\left(\frac{q-1}{7-3q}+\frac{1}{2},\frac{1}{2}\right)-\frac{4(5-p)(p-1)}{(7-3q)^2}B\left(\frac{p-1}{7-3q}+\frac{1}{2},\frac{1}{2}\right).\label{J_0beta}
\end{align}
Now, for fixed $q<\frac{7}{3}$, consider the function $h(s)=(4-s)sB(\frac{s}{7-3q}+\frac{1}{2},\frac{1}{2})$ for $s\in (0,q-1)$. By Lemma \ref{lem:Bxineq},
\begin{align*}
h'(s) &= (4-2s)B(\frac{s}{7-3q}+\frac{1}{2},\frac{1}{2})+(4-s)\frac{s}{7-3q}\frac{\partial B}{\partial x}(\frac{s}{7-3q}+\frac{1}{2},\frac{1}{2})\\
&> B(\frac{s}{7-3q}+\frac{1}{2},\frac{1}{2})(2-\frac{3s}{2})>0
\end{align*}
Hence (\ref{J_0beta}) is positive for any $p\in (1,q)$. As $\phi(c)\geq 0$, this shows that {\crm{$\phi(c)> 0$ and}} the integral of the first term is positive. For the second term, we similarly get a $c\in(0,1]$ such that $N_2(s)\leq0$ for $s\in[0,c)$ and $N_2(s)>0$ for $s\in (c,1)$. Using the positive decreasing function $\phi(s) = \frac{\left(D_1(s)+\beta D_2(s)\right)^\frac{3}{2}}{(D_2(s))^\frac{3}{2}}$, we get
\begin{align*}
\phi(c)\int_0^1\frac{N_2}{(D_1+\beta D_2)^\frac{3}{3}}ds>\int_0^1\frac{(5-r)(1-s^\frac{r-1}{2})-(5-q)(1-s^\frac{q-1}{2})}{(s^\frac{q-1}{2}-s^\frac{r-1}{2})^\frac{3}{2}}ds
\end{align*}
By Lemma \ref{lem:2pow}, the right hand side is positive for $2q+r<7$. Hence both terms are positive, and hence $J(0,\gamma)>0$.
\end{proof}

\begin{proposition}\label{prop:DFJ0neg}
If $2p+q>7$, then $J(0,\gamma)<0$ for all $\gamma\in \R$.
\end{proposition}
\begin{proof}
First suppose $p\geq\frac{7}{3}$, and consider the Iliev-Kirchev formula (\ref{I-KJ}). For $p\geq\frac{7}{3}$ and $\omega=0$, $U(s)/s=o(s^\frac{2}{3})$. Since $U'(a(0,\gamma))<0$, we have $U'(a(\omega,\gamma))-U'(s)<0$ on a neighbourhood of $0$ for $\omega$ sufficiently small. Since the integrand is uniformly integrable away from $0$, we then have
\begin{align*}
J(\omega,\gamma)=\frac{-1}{2U'(a)}\int_0^a\frac{3\sqrt{s}}{\sqrt{U(s)}}+\frac{U'(a)-U'(s)}{\left(U(s)/s\right)^\frac{3}{2}}ds\to-\infty
\end{align*}
as $\omega\to 0$. Now suppose $p<\frac{7}{3}$. As in the stable case, we consider each term $N_1, N_2$ in
\begin{align*}
J(0,\gamma)=C\left(\frac{p+1}{a^\frac{p-1}{2}}\right)^\frac{1}{2}\int_0^1\frac{N_1(s)+\beta N_2(s)}{(D_1(s)+\beta D_2(s))^\frac{3}{2}}ds
\end{align*}
separately. We have $N_2(0)<0$ and $N_2(1) =0$. By Lemma \ref{lem:RoS}, $N_2$ cannot have three positive zeros, so there is a $c\in (0,1]$ such that $\frac{N_2(s)}{(D_1(s)+\beta D_2(s))^\frac{3}{2}}\leq0$ for $s\in [0,c)$ and $\frac{N_2(s)}{(D_1(s)+\beta D_2(s))^\frac{3}{2}}>0$ for $s\in (c,1)$. For $s\in (0,1]$, let $\phi(s) = \frac{(D_1(s)+\beta D_2(s))^\frac{3}{2}}{(s^\frac{p-1}{2}-s^{3-p})^\frac{3}{2}}$ with $\phi(1)=\lim_{s\to 1^-}\phi(s)$. Since $p<\frac{7}{3}$, $\phi(1)<\infty$. Since $3-p<\frac{q-1}{2}$ we see that, by Lemma \ref{lem:lineq}
\begin{align*}
\phi'(s) = \frac{3}{2}\left(\phi(s)\right)^\frac{1}{3}\left(\frac{d}{ds}\frac{s^\frac{p-1}{2}-s^\frac{q-1}{2}}{s^\frac{p-1}{2}-s^{3-p}}+\beta\frac{d}{ds}\frac{s^\frac{q-1}{2}-s^\frac{r-1}{2}}{s^\frac{p-1}{2}-s^{3-p}}\right)>0
\end{align*}
and hence
\begin{align*}
\phi(c)\int_0^1\frac{N_2}{(D_1+\beta D_2)^\frac{3}{2}}ds&<\int_0^1\frac{N_2}{(D_1+\beta D_2)^\frac{3}{2}}\phi(s)ds\\
&=\int_0^1\frac{(5-r)(1-s^\frac{r-1}{2})-(5-q)(1-s^\frac{q-1}{2})}{(s^\frac{p-1}{2}-s^{3-p})^\frac{3}{2}}ds
\end{align*}
with $\phi(c)>0$. Using a change of variables $t=s^\frac{7-3p}{2}$, we write this integral as
\begin{align*}
&\frac{2}{7-3p}\int_0^1\frac{(5-r)(1-t^\frac{r-1}{7-3p})-(5-q)(1-t^\frac{q-1}{7-3p})}{t^\frac{1}{2}(1-t)^\frac{3}{2}}dt\\
&=\frac{2(5-r)}{7-3p}H\left(\frac{1}{2},\frac{r-1}{7-3p}\right)-\frac{2(5-q)}{7-3p}H\left(\frac{1}{2},\frac{q-1}{7-3p}\right)\\
\end{align*}
and using Lemma \ref{lem:BvH}, this becomes
\begin{align}
\frac{4(5-r)(r-1)}{(7-3p)^2}B\left(\frac{r-1}{7-3p}+\frac{1}{2},\frac{1}{2}\right)-\frac{4(5-q)(q-1)}{(7-3p)^2}B\left(\frac{q-1}{7-3p}+\frac{1}{2},\frac{1}{2}\right)\label{J_0betaneg}
\end{align}
Now, for fixed $r>\frac{7}{3}$, consider the function $h(s)=(4-s)sB(\frac{s}{7-3p}+\frac{1}{2},\frac{1}{2})$ for $s\in (q-1,r-1)$. As $q>7-2p$, we have $s>6-2p$. %
 {\crm By Lemma \ref{lem:Bxineq}, with $\mu=1$ for $s\le 4$ and $\mu=0$ for $s>4$,
\begin{align*}
h'(s) &= (4-2s)B(\frac{s}{7-3p}+\frac{1}{2},\frac{1}{2})+(4-s)\frac{s}{7-3p}\frac{\partial B}{\partial x}(\frac{s}{7-3p}+\frac{1}{2},\frac{1}{2})\\
&<\alpha B(\frac{s}{7-3p}+\frac{1}{2},\frac{1}{2}), \qquad \alpha:=4-2s-(4-s)\frac{s}{2s+(7-3p)\mu}.
\end{align*}
For $s>4$, 
$
\alpha=4-2s-(4-s)/2<0$.
For $s<4$, we have 
$\alpha=4-2s-(4-s)\frac{s}{2s+7-3p}$.
If $s\in [2,4)$, then $\alpha<4-2s<0$. If $s \in (4/3,2)$, then $q\le s+1 <3$ and $p>\frac 12(7-q)>2$. But there is no $p\in (2,7/3)$ 
such that $\alpha \ge 0$, which is equivalent to 
$ 2s+7-3p \ge \frac {4s-s^2}{4-2s}$, since
\[
2s+7- \frac {4s-s^2}{4-2s}< 3(3-\frac s2)<3p
\]
for $s \in (4/3,2)$.
Hence $\alpha<0$ when $s \in (4/3,2)$. We have shown $\alpha<0$ and}
 $h'(s)<0$ for $s\in (q-1,r-1)$, hence \eqref{J_0betaneg} is negative. This shows that the integral of the second term is negative. For the first term, we similarly have a $c\in (0,1]$ such that $N_1(s)<0$ for $s\in [0,c)$ and $N_1(s)>0$ for $s\in (c,1]$. Using the increasing function $\phi(s) = \frac{\left(D_1(s)+\beta D_2(s)\right)^\frac{3}{2}}{(D_1(s))^\frac{3}{2}}$ we get
\begin{align*}
\phi(c)\int_0^1\frac{N_1}{(D_1+\beta D_2)^\frac{3}{3}}ds <\int_0^1\frac{(5-q)(1-s^\frac{q-1}{2})-(5-p)(1-s^\frac{p-1}{2})}{(s^\frac{p-1}{2}-s^\frac{q-1}{2})^\frac{3}{2}}ds
\end{align*}
By Lemma \ref{lem:2pow}, the right hand side is negative for $2p+q>7$. Hence both terms are negative, and hence $J(0,\gamma)<0$.
\end{proof}
\begin{proposition}\label{prop:DFJ0US}
    If $2p+q<7<2q+r$, then $J(0,\gamma)<0$ for sufficiently large $\gamma$ and $J(0,\gamma)>0$ for sufficiently large $-\gamma$.
\end{proposition}
\begin{proof}
    Since $p<\frac{7}{3}$, the integral (\ref{pqrstab}) is uniformly integrable, and so $\lim_{\omega\to 0} J(\omega,\gamma) = J(0,\gamma)$ with $J(0,\gamma)$ given by (\ref{J_0}). As $\gamma\to\infty$, $a(0,\gamma)\to \infty$, and so $\beta\to \infty$. Thus,
    \begin{align*}
        J(0,\gamma) = C(0,\gamma)\left(\frac{p+1}{a(0,\gamma)^\frac{p-1}{2}\beta}\right)^\frac{1}{2}\left(\int_0^1\frac{N_2(s)}{(D_2(s))^\frac{3}{2}}+o(1)\right)
    \end{align*}
    where $\int_0^1\frac{N_2(s)}{D_2(s)^\frac{3}{2}}< 0$ by Lemma \ref{lem:2pow}. Hence $J(0,\gamma)<0$ for large $\gamma$. Since $\beta\to 0$ as $\gamma \to -\infty$, we similarly have $J(0,\gamma)>0$ for large $-\gamma$.
\end{proof}
Proposition \ref{prop:DFlimits} shows that there is a stable region when $q<5$. The converse also holds.
\begin{proposition}\label{prop:DFallunStab}
If $q\geq 5$, then $J(\omega,\gamma)<0$ for all $\omega>0$, $\gamma\in \R$.
\end{proposition}
\begin{proof}
For $\gamma\in\R$, let $a_0(\gamma) = a(0,\gamma)>0$. Then $a_0(\gamma)$ satisfies
\begin{align}
    \gamma\frac{a_0(\gamma)^\frac{q-1}{2}}{q+1}=-\frac{a_0(\gamma)^\frac{p-1}{2}}{p+1}+\frac{a_0(\gamma)^\frac{r-1}{2}}{r+1}\label{gamelim}
\end{align}
First suppose $p\geq 5$. Using {\crm $1-s^t \le 1-s^u$ for $s\in(0,1)$, $t<u$},
we have, for $\gamma >0$
\begin{align*}
\frac{N(a,s)}{1-s^\frac{r-1}{2}}&\leq-\frac{5-p}{p+1}a^\frac{p-1}{2}-\gamma\frac{5-q}{q+1}a^\frac{q-1}{2}+\frac{5-r}{r+1}a^\frac{r-1}{2}\\
\end{align*}
and for $\gamma<0$
\begin{align*}
\frac{N(a,s)}{1-s^\frac{p-1}{2}}&\leq-\frac{5-p}{p+1}a^\frac{p-1}{2}-\gamma\frac{5-q}{q+1}a^\frac{q-1}{2}+\frac{5-r}{r+1}a^\frac{r-1}{2}.\\
\end{align*}
In the case $a=a_0(\gamma)$, (\ref{gamelim}) yields
\begin{align*}
    -\frac{5-p}{p+1}a_0(\gamma)^\frac{p-1}{2}-\gamma\frac{5-q}{q+1}a_0(\gamma)^\frac{q-1}{2}+\frac{5-r}{r+1}a_0(\gamma)^\frac{r-1}{2}&=-\frac{q-p}{p+1}a_0(\gamma)^\frac{p-1}{2}-\frac{r-q}{r+1}a_0(\gamma)^\frac{r-1}{2}<0.
\end{align*}
Hence $N(a_0(\gamma),s)<0$ for all $\gamma\in \R$ and $s\in (0,1)$. Since $r>5$, $N(a,s)$ is also negative for large $a$. By Lemma \ref{lem:RoS}, $N(a,s)$ changes sign at most once for $a\in (0,\infty)$, so $N(a,s)<0$ for all $a>a_0(\gamma)$.

Now suppose $p<5$. If $\gamma<0$, then each term in $N(a,s)$ is negative for all $a>a_0(\gamma)$ and $s\in (0,1)$. If $\gamma>0$, then
\begin{align}
    \frac{N(a,s)}{1-s^\frac{r-1}{2}}&\leq -\gamma\frac{5-q}{q+1}a^\frac{q-1}{2}+\frac{5-r}{r+1}a^\frac{r-1}{2}\label{L23}.
\end{align}
The right hand side is negative for $a=a_0(\gamma)$, as
\begin{align*}
    -\gamma\frac{5-q}{q+1}a_0(\gamma)^\frac{q-1}{2}+\frac{5-r}{r+1}a_0(\gamma)^{\frac{r-1}2}=\frac{5-q}{p+1}a_0(\gamma)^\frac{p-1}{2}-\frac{r-q}{r+1}a_0(\gamma)^\frac{r-1}{2}<0.
\end{align*}
The right hand side of (\ref{L23}) is also negative for large $a$ as $r>5$, and is therefore negative for all $a>a_0(\gamma)$ by Lemma \ref{lem:RoS}. Hence $N(a,s)<0$ for all $a>a_0(\gamma)$, and hence $J(\omega,\gamma)<0$ for all $\omega>0$ and $\gamma\in\R$.
\end{proof}
The existence of an unstable region is harder to determine for given $1<p<q<r$. When $2q+r\leq 7$, we know by Propositions \ref{prop:DFlimits} and \ref{prop:DFJ0pos} that $J(\omega,\gamma)$ is eventually positive in each limit case $\omega\to 0$, $\omega\to \infty$, and $\gamma\to \pm\infty$. This leads us expect that $J(\omega,\gamma)>0$ for all $\omega>0$ and $\gamma\in\R$ when {\crm $2q+r\leq 7$}. This is supported by numerical observations in the previous section.

\section{Theorems for the DD Case}\label{S9}
{\crm In this section we prove results for the DD case with $a_1=a_3=-1$.}

In the DD case, solutions do not exist for large $\gamma$ and large $\omega$, (they exist for $\omega<\omega^*(\gamma)$ and $\gamma<\gamma_1$ by Proposition \ref{prop:Rex}.)
and the limits of $J(\omega,\gamma)$ close to the nonexistence curve $\Gamma_{\nex}$ are given by Proposition \ref{prop:Gamlimits}. The limits of $J(\omega,\gamma)$ for $\gamma\to - \infty$ are proved in the same way as the FF case.
\begin{proposition}\label{prop:DDlimits}
In the DD case,
\begin{enumerate}
\item If $q<5$, then $\lim_{\gamma\to-\infty}J(\omega,\gamma)=0^+$ for all $\omega>0$. 
\item If $q\geq 5$, then $\lim_{\gamma\to-\infty} J(\omega,\gamma)= 0^-$ for all $\omega>0$.
\end{enumerate}
\end{proposition}
The limits for small $\omega$ are only partially described by the following proposition.
\begin{proposition}\label{prop:DDJ0limits}
In the DD case,
    \begin{enumerate}
        \item
        If $p\geq \frac{7}{3}$, then $\lim_{\omega\to 0} J(\omega,\gamma)=-\infty$ for all $\gamma<\gamma_1$.
        \item 
        Suppose $p<\frac{7}{3}$. Then $\lim_{\omega\to 0} J(\omega,\gamma) = J(0,\gamma)$ for all $\gamma\in \R$, and the following hold
        \begin{enumerate}
            \item $\lim_{\gamma\to \gamma_1^-}J(0,\gamma)=\infty$
            \item If $2p+q<7$, then $\lim_{\gamma\to -\infty} J(0,\gamma) =0^+$.
            \item If $2p+q> 7$, then $\lim_{\gamma\to -\infty} J(0,\gamma) = 0^-$.
        \end{enumerate}
    \end{enumerate}
\end{proposition}
\begin{proof}  
Case 1: $p\geq \frac{7}{3}$.
        Let $\gamma<\gamma_1$, and let $\crm a_0(\gamma) \nc =\lim_{\omega\to 0} a(\omega,\gamma)$. By
\crm        Lemmas \ref{lem:aomega}, \ref{lem:agamma},  and \nc
         Proposition \ref{prop:Rex}, $0<a_0(\gamma)<a(0,\gamma_1)=\crm\big(\frac{(r+1)(q-p)}{(p+1)(r-q)}\big)^{\frac 2{r-p}}\nc$. As $s\to 0$,
        \begin{align*}
            N(a_0(\gamma),s)&\to-\frac{5-p}{p+1}a_0(\gamma)^\frac{p-1}{2}-\gamma\frac{5-q}{q+1}a_0(\gamma)^\frac{q-1}{2}-\frac{5-r}{r+1}a_0(\gamma)^\frac{r-1}{2}\\
            &=-\frac{q-p}{p+1}a_0(\gamma)^\frac{p-1}{2} \crm + \nc \frac{r-q}{r+1}a_0(\gamma)^\frac{r-1}{2}
            \\
            &\crm <a_0(\gamma)^\frac{p-1}{2} \left (-\frac{q-p}{p+1}+  \frac{r-q}{r+1}a_0(0,\gamma_1)^\frac{r-p}{2} \right)=    0. \nc
        \end{align*}
        Moreover, the convergence $N(a(\omega,\gamma))\to N(a_0(\gamma))$ as $\omega\to 0$ is uniform on a neighbourhood of $0$. Hence, there is a $\delta>0$ such that $N(a,s)\leq 0$ for $s\in [0,\delta]$ and $\omega$ close to $0$. By Fatou's Lemma,
        \begin{align}
            \lim_{a\to a_0(\gamma)}\int_0^\delta \frac{N(a,s)}{(D(a,s))^\frac{3}{2}}ds\leq \int_0^\delta \frac{N(a_0(\gamma),s)}{(D(a_0(\gamma),s))^\frac{3}{2}}ds\label{inflim}
        \end{align}
        and
        \begin{align*}
            D(a_0(\gamma),s) = (s^\frac{p-1}{2}-s^\frac{q-1}{2})\frac{a_0(\gamma)^\frac{p-1}{2}}{p+1}+(s^\frac{r-1}{2}-s^\frac{q-1}{2})\frac{a_0(\gamma)^\frac{r-1}{2}}{r+1}=O(s^\frac{p-1}{2})=O(s^\frac{2}{3})
        \end{align*}
        so the right side of (\ref{inflim}) is $-\infty$. For $a$ close to $a_0(\gamma)$ and $s\in[\delta,1]$ the $\frac{D(a,s)}{1-s}$ is bounded away from $0$, so the integrand is uniformly integrable on $[\delta,1]$. Hence $J(\omega,\gamma)\to -\infty$ as $\omega \to 0$.

\smallskip
Case 2: $p< \frac{7}{3}$.

For Part (a):
            As $p<\frac{7}{3}$, $\lim_{\omega\to 0} J(\omega,\gamma) = J(0,\gamma)$ with $J(0,\gamma)$ given by (\ref{J_0}). By Proposition \ref{prop:Rex}, $a(0,\gamma)^\frac{r-p}{2}<a(0,\gamma_1)^\frac{r-p}{2}=\frac{(r+1)(q-p)}{(p+1)(r-q)}$ for $\gamma<\gamma_1$, and so $\beta(\gamma)<\frac{q-p}{r-q}$. When $\beta=\frac{q-p}{r-q}$, the denominator
    \begin{align*}
        D_1(s)+\beta D_2(s)=(s^\frac{p-1}{2}-s^\frac{q-1}{2})-\beta (s^\frac{q-1}{2}-s^\frac{r-1}{2})
    \end{align*}
    has a double zero at $s=1$. For the numerator we have,
    \begin{align*}\hspace{-1cm}
        \frac{\crm (r-q) \nc N_1(s)+(q-p) N_2(s)}{1-s^\frac{1}{2}}\to -(5-p)(p-1)(r-q)+(5-q)(q-1)(r-p)-(5-r)(r-1)(q-p)
    \end{align*}
    as $s\to 1$, which is positive by the concavity of $s\mapsto (5-s)(s-1)$. Since 
    \begin{align*}
        (r-q)\frac{N_1(s)+\beta N_2(s)}{1-s^\frac{1}{2}}\to \frac{N_1(s)+(q-p) N_2(s)}{1-s^\frac{1}{2}}
    \end{align*}
\crm as $\gamma \to \gamma_1^-$    
    uniformly in $s$, \nc there is a $\delta>0$ such that $N_1(s)+\beta N_2(s)>0$ for $s\in (1-\delta,1]$ and $\beta$ close to $\frac{q-p}{r-q}$. Since the numerator is $\Theta(1-s)$ near $1$ when $\beta = \frac{q-p}{r-q}$, we get 
    \begin{align*}
        \int_{1-\delta}^1\frac{N_1(s)+\beta N_2(s)}{(D_1(s)+\beta D_2(s))^\frac{3}{2}}ds \to \infty\quad \text{as } \beta\to \frac{q-p}{r-q}.
    \end{align*}
    Since the integrand is uniformly integrable on $[0,1-\delta]$, this shows that $\lim_{\gamma\to \gamma_1^-}J(0,\gamma)=\infty$.

\crm For Parts (b) and (c): \nc
    As $\gamma\to-\infty$, $\beta\to 0$. Taking $\beta\to 0$ in (\ref{J_0}), gives
    \begin{align*}
        J(0,\gamma) = C(0,\gamma)\left(\frac{p+1}{a(0,\gamma)^\frac{p-1}{2}}\right)^\frac{1}{2}\left(\int_0^1\frac{N_1(s)}{(D_1(s))^\frac{3}{2}}+o(1)\right)
    \end{align*}
    By Lemma \ref{lem:2pow}, $\int_0^1\frac{N_1(s)}{D_1(s)^\frac{3}{2}}$ is negative when $2p+q>7$, and positive when $2p+q<7$. 
\end{proof}
Since both $\lim_{\gamma\to \gamma_1} J(0,\gamma)=\infty$ and $\lim_{\gamma\to -\infty} J(0,\gamma)=0^+$ when $2p+q\leq 7$, we expect that $J(0,\gamma)>0$ for all $\gamma\in \R$ in this case. If this is true, then all limits of $J(\omega,\gamma)$ near $\Gamma_{\nex}$, for $\omega\to 0$, and for $\gamma\to -\infty$ are all positive when $2p+q<7$, which in turn suggests that $J(\omega,\gamma)>0$ for all $\omega>0$, $\gamma\in\R$ when $2p+q<7$. This is supported by numerical observations in the  Section \ref{S2}.

\section*{Acknowledgments}
We thank Stefan Le Coz and the anonymous referees for very helpful suggestions.
The research of both TM and TT was partially supported by the NSERC grant
RGPIN-2023-04534. TM was also supported by an NSERC USRA.

\addcontentsline{toc}{section}{\protect\numberline{\hspace{2mm}}{References}}
\bibliographystyle{abbrv}
\bibliography{357nls}
\end{document}